\documentclass[12pt]{amsart}
\usepackage{amsmath,amsfonts,amsthm,amssymb,xypic}

\newcommand{\scrC}{\mathcal{C}}
\newcommand{\scrD}{\mathcal{D}}

\newcommand{\scrL}{\mathcal{L}}

\newcommand{\scrO}{\mathcal{O}}
\newcommand{\scrP}{\mathcal{P}}

\newcommand{\bC}{\mathbf{C}}
\newcommand{\bD}{\mathbf{D}}
\newcommand{\bE}{\mathbf{E}}

\newcommand{\bI}{\mathbf{I}}

\newcommand{\bP}{\mathbf{P}}

\newcommand{\bR}{\mathbf{R}}

\newcommand{\bbC}{\mathbb{C}}

\newcommand{\bbP}{\mathbb{P}}
\newcommand{\bbR}{\mathbb{R}}
\newcommand{\bbZ}{\mathbb{Z}}

\newcommand{\sfN}{\mathsf{N}}
\newcommand{\Sh}{\mathrm{Sh}}
\newcommand{\Hom}{\mathrm{Hom}}
\newtheorem{dummy}{Dummy}[section]
\newtheorem{proposition}[dummy]{Proposition}
\newtheorem{lemma}[dummy]{Lemma}
\newtheorem{corollary}[dummy]{Corollary}
\newtheorem{theorem}[dummy]{Theorem}
\theoremstyle{definition}
\newtheorem*{maintheorem}{Main Theorem}
\newtheorem{definition}[dummy]{Definition}
\newtheorem{remark}[dummy]{Remark}
\newtheorem{example}[dummy]{Example}
\newcommand{\Sets}{\mathbf{Sets}}
\newcommand{\Cat}{\mathbf{Cat}}

\newcommand{\Funct}{\mathrm{Funct}}
\newcommand{\twofunct}{2\mathrm{Funct}}
\newcommand{\twonat}{2\mathrm{Nat}}
\newcommand{\twomon}{2\mathrm{Mon}}
\newcommand{\twoexit}{2\mathrm{Exitm}}
\newcommand{\St}{\mathrm{St}}
\newcommand{\Prest}{\mathrm{Prest}}
\newcommand{\twolim}{2\!\varprojlim}
\newcommand{\twocolim}{2\!\varinjlim}
\newcommand{\wtwolim}[2]{\underset{#1}{\operatorname{2\!}\varprojlim}\,{#2}}
\newcommand{\wtwocolim}[2]{\underset{#1}{\operatorname{2\!}\varinjlim}\,{#2}}
\newcommand{\shHom}{\underline{\mathrm{Hom}}}
\newcommand{\cat}{\mathbf{cat}}
\newcommand{\gpd}{\mathbf{gpd}}
\newcommand{\res}{\mathit{res}}
\newcommand{\orpi}{\overrightarrow{\pi}}
\newcommand{\Strat}{\mathbf{Strat}}
\newcommand{\twocat}{\mathbf{2cat}}
\newcommand{\tame}{\mathit{tame}}
\newcommand{\tetrah}[4]{\xymatrix{#1 \ar[r] \ar[dr] \ar[d] & #2 \ar[d] 
\ar @{} [dr] |{=} & #1 \ar[r] \ar[d] & #2 \ar[d] \ar[dl] \\ #3 \ar[r] & #4 
& #3 \ar[r] & #4}}

\title{Exit paths and constructible stacks.}
\date{August 2007}
\author{David Treumann}

\begin{document}

\maketitle

\begin{abstract}
For a Whitney stratification $S$ of a space $X$ (or more generally a topological stratification in the sense of Goresky and MacPherson) we introduce the notion of an $S$-constructible stack of categories on $X$.  The motivating example is the stack of $S$-constructible perverse sheaves.  We introduce a 2-category $EP_{\leq 2}(X,S)$, called the exit-path 2-category, which is a natural stratified version of the fundamental 2-groupoid.  Our main result is that the 2-category of $S$-constructible stacks on $X$ is equivalent to the 2-category of 2-functors $\twofunct(EP_{\leq 2}(X,S),\Cat)$ from the exit-path 2-category to the 2-category of small categories.
\end{abstract}

\section{Introduction}

This paper is concerned with a generalization of the following well-known and very old theorem:

\begin{theorem}
\label{intro1}
Let $X$ be a connected, locally contractible topological space.  The category of locally constant sheaves 
of sets on $X$ is equivalent to the category of $G$-sets, where $G$ is the fundamental group of $X$.
\end{theorem}

We wish to generalize this theorem in two directions.  In one direction we will consider sheaves which are not necessarily locally constant -- namely, constructible sheaves.  In the second direction we will consider sheaves of ``higher-categorical'' objects -- these generalizations of sheaves are usually called \emph{stacks}.  Putting these together, we get the ``constructible stacks'' of the title.  In this paper, we introduce an object -- the \emph{exit-path 2-category} -- which plays for constructible stacks the same role the fundamental group plays for locally constant sheaves.

\subsection{Exit paths and constructible sheaves}

A sheaf $F$ on a space $X$ is called ``constructible'' if the space may be decomposed into suitable pieces with $F$ locally constant on each piece.  To get a good theory one needs to impose some conditions on the decomposition -- for our purposes the notion of a \emph{topological stratification}, introduced in \cite{ih2}, is the most convenient.  A topological stratification $S$ of $X$ is a decomposition of $X$ into topological manifolds, called ``strata,'' which are required to fit together in a nice way.  (Topological stratifications are more general than Whitney stratifications and Thom-Mather stratifications \cite{thom}, \cite{mather}, and they are less general than Siebenmann's CS sets \cite{siebenmann} and Quinn's manifold stratified spaces \cite{quinn}.  A precise definition is given in \cite{ih2} and in section \ref{sec-cs}.)  A sheaf is called \emph{$S$-constructible} if it is locally constant along each stratum of $(X,S)$.

MacPherson observed (unpublished) that, for a fixed stratification $S$ of $X$, it is possible to give a description of the $S$-constructible sheaves on $X$ in terms of monodromy along certain paths.  A path $\gamma:[0,1] \to X$ has the \emph{exit property} with respect to $S$ if, for each $t_1 \leq t_2 \in [0,1]$, the dimension of the stratum containing $\gamma(t_1)$ is less than or equal to the dimension of the stratum containing $t_2$.  Here is a picture of an exit path in the plane, where the plane is stratified by the origin, the rays of the axes, and the interiors of the quadrants:

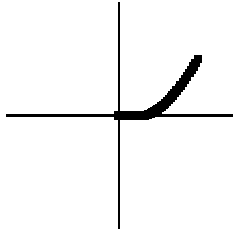
\begin{center}
\setlength{\unitlength}{3cm}
\begin{picture}(1,1)

\put(.5,0){\line(0,1){1}}
\put(0,.5){\line(1,0){1}}

\linethickness{1mm}
\put(.48,.5){\line(1,0){.12}}
\put(.6,.5){\qbezier(0,0)(.1,0)(.25,.25)}

\end{picture}
\end{center}

The concatenation of two exit paths is an exit path, and passing to homotopy classes yields a category we call $EP_{\leq 1}(X,S)$.  That is, the objects of $EP_{\leq 1}(X,S)$ are points of $X$, and the morphisms are homotopy classes of exit paths between points.  (We require that the homotopies $h:[0,1] \times [0,1] \to X$ have the property that each $h(t,-)$ is an exit path, and that $h$ does not intersect strata in a pathological way.  We conjecture that the latter ``tameness'' condition can be removed.)  Then we have the following analog of theorem \ref{intro1}:

\begin{theorem}[MacPherson]
\label{intro2}
Let $(X,S)$ be a topologically stratified space.  The category of $S$-constructible sheaves of sets is equivalent to the category $\Funct\big( EP_{\leq 1}(X,S),\Sets\big)$
of $\Sets$-valued functors on $EP_{\leq 1}(X,S)$.
\end{theorem}

\begin{example}
\label{introex1}
Let $D$ be the open unit disk in the complex plane, and let $S$ be the stratification by the origin $\{0\}$ and its complement $D-\{0\}$.  Then $EP_{\leq 1}(D,S)$ is equivalent to a category with two objects, one labeled by $0$ and one labeled by some other point $x \in D - \{0\}$.  The arrows in this category are generated by arrows $\alpha: 0 \to x$ and $\beta:x \to x$; $\beta$ generates the automorphism group $\bbZ$ and we have $\beta \circ \alpha = \alpha$.  The map $\alpha$ is represented by an exit path from $0$ to $x$ in $D$ and the path $\beta$ is represented by a loop around $0$ based at $x$.

It follows that an $S$-constructible sheaf (of, say, complex vector spaces) on $D$ is given by two vector spaces $V$ and $W$, a morphism $a:V \to W$ and a morphism $b:W \to W$, with the property that $b$ is invertible and that $b \circ a = a$.
\end{example}

\begin{example}
\label{introex2}
Let $\bbP^1 = \bbC \cup \infty$ be the Riemann sphere, and let $S$ be the stratification of $\bbP^1$ by one point $\{\infty\}$ and the complement $\bbC$.  Then $EP_{\leq 1}(\bbP^1,S)$ is equivalent to a category with two objects, one labeled by $\infty$ and the other labeled by $0 \in \bbC$.  The only nontrivial arrow in this category is represented by an exit path from $\infty$ to $0$; all such paths are homotopic to each other.

It follows that an $S$-constructible sheaf on $\bbP^1$ is given by two vector spaces $V$ and $W$, and a single morphism $V \to W$.
\end{example}

\subsection{Perverse sheaves}
\label{introsection2}

Let $(X,S)$ be a topologically stratified space.  For each function $p:S \to \bbZ$ from connected strata of $(X,S)$ to integers, there is an abelian category $\bP(X,S,p)$ of ``$S$-constructible perverse sheaves on $X$ of perversity $p$,'' introduced in \cite{bbd}.  It is a full subcategory of the derived category of sheaves on $X$; its objects are complexes of sheaves whose cohomology sheaves are $S$-constructible, and whose derived restriction and corestriction to strata satisfy certain cohomology vanishing conditions depending on $p$.  It is difficult to lay hands on the objects and especially the morphisms of $\bP(X,S,p)$, although we get $\Sh_S(X)$ as a special case when $p$ is constant.

There is a small industry devoted to finding concrete descriptions of the category $\bP(X,S,p)$ in terms of ``linear algebra data,'' similar to the description of $S$-constructible sheaves given in the examples above (\cite{macphersonvilonen}, \cite{gmv}, \cite{bradengrinberg}, \cite{braden}, \cite{vybornov}).  Here $(X,S)$ is usually a complex analytic space, with complex analytic strata, and $p$ is the ``middle perversity'' which associates to each stratum its complex dimension.  The first example was found by Deligne:

\begin{example}
\label{introex3}
Let $D$ and $S$ be as in example \ref{introex1}.  The category $\bP(X,S,p)$, where $p$ is the middle perversity, is equivalent to the category of tuples $(V,W,m,n)$, where $V$ and $W$ are $\bbC$-vector spaces, $m:V \to W$ and $n:W \to V$ are linear maps, and $1_W-mn$ is invertible.
\end{example}

A topological interpretation of this description was given by MacPherson and Vilonen \cite{macphersonvilonen}.  If $(L,S)$ is a compact topologically stratified space then the open cone $CL$ on $L$ has a natural topological stratification $T$, in which the cone point is a new stratum.  MacPherson and Vilonen gave a description of perverse sheaves on $(CL,T)$ in terms of perverse sheaves on $L$, generalizing Deligne's description \ref{introex3}.  

One of the important properties of perverse sheaves is that they form a \emph{stack}; it means that a perverse sheaf on a space $X$ may be described in the charts of an open cover of $X$.  A topologically stratified space has an open cover in which the charts are of the form $CL\times \bbR^k$.  The stack property together with the MacPherson-Vilonen construction give an inductive strategy for computing categories of perverse sheaves.  One of the motivations for the theory in this paper is to analyze this strategy systematically; see \cite{treumann}.

\subsection{Constructible stacks}

In this paper, we introduce the notion of a \emph{constructible stack} on a topologically stratified space.  Our main example is the stack $\scrP$ of $S$-constructible perverse sheaves discussed in section \ref{introsection2}.  Our main result is a kind of classification of constructible stacks, analogous to the description of constructible sheaves by exit paths.

\begin{maintheorem}[Theorem \ref{finaltheorem}]
\label{maintheorem}
Let $(X,S)$ be a topologically stratified space.  There is a 2-category $EP_{\leq 2}(X,S)$, introduced in section \ref{sec7}, such that the 2-category of $S$-constructible stacks on $X$ is equivalent to the 2-category of 2-functors $\mathrm{2Funct}(EP_{\leq 2}(X,S),\mathbf{Cat})$.
\end{maintheorem}

The appearance of 2-categories in this theorem is an application of a well-known philosophy of Grothendieck \cite{pursuingstacks}.  It is a modification of theorems in \cite{polesellowaschkies} and \cite{toen}, where it was shown that \emph{locally constant} stacks on $X$ correspond to representations of higher \emph{groupoids}, namely the groupoid of points, paths, homotopies, homotopies between homotopies, and so on in $X$.  $EP_{\leq 2}(X)$ is a 2-truncated, stratified version of this: the objects are the points of $X$, the morphisms are exit paths, and the 2-morphisms are homotopy classes of homotopies between exit paths.  (Once again, we require a tameness condition on our homotopies and also our homotopies between homotopies.)

\begin{example}
Let $\bbP^1$ and $S$ be as in example \ref{introex2}.  Then $EP_{\leq 2}(\bbP^1,S)$ is equivalent to a 2-category with two objects, labeled by $\infty$ and $0$ as before, and one arrow from $\infty$ to $0$ represented by an exit path $\alpha$. The group of homotopies from $\alpha$ to itself is $\bbZ$, generated by a homotopy that rotates $\alpha$ around the 2-sphere once.

It follows that an $S$-constructible stack on $\bbP^1$ is given by a pair $\bC_{\infty}$ and $\bC_0$ of categories, a  functor $\alpha:\bC_{\infty} \to \bC_0$, and a natural automorphism $f:\alpha \to \alpha$.  For the stack of $S$-constructible perverse sheaves, $\bC_0$ is the category of vector spaces, $\bC_{\infty}$ is Deligne's category described in example \ref{introex3}, $\alpha$ is the forgetful functor $\alpha:(V,W,m,n) \mapsto W$, and $f$ is the map $1_W - mn$, which is invertible by assumption.
\end{example}

\subsection{Notation and conventions}

$\bbR$ denotes the real numbers.  For $a,b \in \bbR$ with $a \leq b$ we use $(a,b)$ to denote the open interval and $[a,b]$ to denote the closed interval between $a$ and $b$.  We use $[a,b)$ and $(a,b]$ to denote half-open intervals.  If $L$ is a compact space then $CL$ denotes the open cone on $L$, that is, the space $L \times [0,1) / L \times \{0\}$.  A $d$-cover of a space $X$ is a collection of open subsets of $X$ that covers $X$ and that is closed under finite intersections.

For us, a ``2-category'' is a strict 2-category in the sense that composition of 1-morphisms is strictly associative.  On the other hand we use ``2-functor'' to refer to morphisms of 2-categories that only preserve composition of 1-morphisms up to isomorphism.  We will refer to sub-2-categories of a 2-category $\bC$ as simply ``subcategories of $\bC$.''  For more basic definitions and properties of 2-categories see appendix B.

All our stacks are stacks of categories.  We write $\Prest(X)$ for the 2-category of prestacks on a space $X$ and $\St(X) \subset \Prest(X)$ for the full subcategory of $\Prest(X)$ whose objects are stacks.  We use $\St_{lc}(X) \subset \St(X)$ to denote the full subcategory of locally constant stacks on $X$, which we introduce
in section \ref{sec-lcstacks}.  When $S$ is a topological stratification (definition \ref{topstratspace}) of $X$ we write $\St_S(X) \subset \St(X)$ for the full subcategory of $S$-constructible stacks on $X$, which we introduce in section \ref{sec-cs}.  For more basic definitions and properties of stacks see appendix A.

\section{Locally constant stacks}
\label{sec-lcstacks}

In this section we introduce locally constant stacks of categories.  A stack is called \emph{constant} if it equivalent to the stackification of a constant prestack, and \emph{locally constant} if this is true in the charts of an open cover.  Our main objective is to give an equivalent definition that is easier to check in practice: on a locally contractible space, a stack $\scrC$ is locally constant if and only if the restriction functor $\scrC(U) \to \scrC(V)$ is an equivalence of categories whenever $V$ and $U$ are contractible.  This is theorem \ref{lcscrit}.  We also develop some basic properties of locally constant stacks, including a base-change result (theorem \ref{serrebasechange}) and the homotopy invariance of the 2-category of locally constant stacks (theorem \ref{lcshotinv}).

\subsection{Constant stacks}  
Let $\bC$ be a small category.  On any space $X$ we have the constant
$\bC$-valued prestack, and its stackification.  We will denote the prestack by $\bC_{p;X}$, and its stackification by $\bC_X$.

\begin{example}
\label{ex-localsystems}
Let $X$ be a locally contractible space, or more generally any space in which each point has a fundamental system of neighborhoods over which each locally constant sheaf is constant.  If $\bC$ is the category of sets, then $\bC_X$ is naturally equivalent to the stack $\scrL\scrC_X$ of locally constant sheaves.  That is, the map 
$\bC_{p;X} \to \scrL\scrC_X$ that takes a set $E \in \bC = \bC_{p;X}(U)$ to the constant sheaf over $U$ with fiber $E$ induces an equivalence $\bC_X \to \scrL\scrC_X$.  Indeed, it induces an equivalence on stalks by the local contractibility of $X$.
\end{example}

\begin{proposition}
\label{constantstacks}
Let $X$ be a locally contractible space and $\bC$ be a small category.
\begin{enumerate}
\item If $F$ and $G$ are objects in $\bC_X(X)$, then the sheaf $\shHom(F,G):U \mapsto \Hom(F\vert_U,G\vert_U)$
is locally constant on $X$.

\item
Let $U \subset X$ be an open set.  For each point $x \in U$, the restriction functor from $\bC = \bC_{p;X}(U)$ to the stalk of $\bC_X$ at $x$ is an equivalence of categories.

\item
Let $U \subset X$ be an open set.  Suppose that $U$ is contractible.  Then for each point $x \in U$, the restriction functor $\bC_X(U) \to \bC_{X,x}$ is an equivalence of categories.

\end{enumerate}

\end{proposition}

\begin{proof}
The map of prestacks $\Hom:\bC_{p;X}^{op} \times \bC_{p;X} \to \Sets_{p;X}$ 
induces a map of stacks
$\bC_X^{op} \times \bC_X \to \Sets_X \cong \scrL\scrC_X$.  The object in 
$\scrL\scrC_X(X)$ associated to a pair $(F,G) \in \big(\bC_X^{op} \times \bC_X\big)(X)$ is exactly the sheaf $\shHom(F,G)$.  This proves the first assertion.

The second assertion is trivial.  To prove the third assertion, note that $\bC_X(U) \to \bC_{X,x}$ is always essentially surjective, since the equivalence $\bC_{p;X}(U) \to \bC_{X,x}$ factors through it.  We therefore only have to show that $\bC_X(U) \to \bC_{X,x}$ is fully faithful.  For objects $F$ and $G$ of $\bC_X(U)$, we have just seen that $\shHom(F,G)$ is locally constant on $U$.  Since $U$ is contractible
$\shHom(F,G)$ is constant, and so $\Hom_{\bC_X(U)}(F,G) = \shHom(F,G)(U) \to \shHom(F,G)_x \cong \Hom(F_x,G_x)$ is a bijection.  This completes the proof.
\end{proof}

\subsection{Locally constant stacks}

\begin{definition}
A stack $\scrC$ on $X$ is called \emph{locally constant} if there exists an open cover $\{U_i\}_{i \in I}$
of $X$ such that $\scrC\vert_{U_i}$ is equivalent to a constant stack.  Let $\St_{lc}(X) \subset \St(X)$
denote the full subcategory of the 2-category of stacks on $X$ whose objects are the locally constant stacks.
\end{definition}

\begin{proposition}
\label{constantstacks2}
Let $X$ be a locally contractible space and let $\scrC$ be a locally constant stack on $X$.  
\begin{enumerate}
\item Let $U \subset X$ be an open set, and let $F,G \in \scrC(U)$.  The sheaf $\shHom(F,G)$ is locally constant on $U$.

\item Every point $x \in X$ has a contractible neighborhood $V$ such that the restriction map $\scrC(V) \to \scrC_x$ is an equivalence of categories.

\end{enumerate}
\end{proposition}

\begin{proof}
Assertion 1 follows directly from assertion 1 of proposition \ref{constantstacks}, and assertion 2 follows directly from assertion 3 of proposition \ref{constantstacks}.
\end{proof}

\begin{proposition}
\label{prop-pullpushlc}
Let $X$ and $Y$ be topological spaces, and let $f:X \to Y$ be a continuous map.  Suppose $\scrC$ is locally constant on $Y$.  Then $f^*\scrC$ is locally constant on $X$.
\end{proposition}

\begin{proof}
If $\scrC$ is constant over the open sets $U_i$ of $Y$, then $f^* \scrC$ will be constant over the open sets $f^{-1}(U_i)$ of $X$.
\end{proof}

The homotopy invariance of $\St_{lc}(X)$ is a consequence of the following base-change result:

\begin{proposition}
\label{basechangelc}
Let $X$ and $Y$ be topological spaces, and let $f:X \to Y$ be a continuous map.  Let $g$ denote the map $(id,f):[0,1] \times X \to [0,1] \times Y$, and let $p:[0,1] \times X \to X$ and $q:[0,1] \times Y \to Y$ be the natural projection maps.
$$\xymatrix{
[0,1] \times X \ar[r]^g \ar[d]_p & [0,1] \times Y \ar[d]^q \\
X \ar[r]_f & Y
}
$$
Let $\scrC$ be a locally constant stack on $[0,1] \times Y$.  Then the base change map $f^* q_* \scrC \to p_* g^* \scrC$ is an equivalence of stacks on $X$.
\end{proposition}

\begin{proof}
Let us first prove the following claim: every point $t \in [0,1]$ has a neighborhood $I \subset [0,1]$ such that $\scrC\big( [0,1] \times Y\big) \to \scrC\big( I \times Y\big)$ is an equivalence of categories.  There is an open cover of $[0,1] \times Y$ of the form $\{I_\alpha \times U_\beta\}$ such that $\scrC$ is constant over each chart $I_\alpha \times U_\beta$.  
By by basic properties of the interval, $\scrC$ is constant over $[0,1] \times U_\beta$, and for any subinterval $I \subset [0,1]$ the restriction map $\scrC\big( [0,1] \times U_\beta \big) \to \scrC(I \times U_\beta)$ is an equivalence of categories.  This implies the claim.  

Let $y$ be a point in $Y$, and let $t$ be a point in $[0,1]$.  Let $\{U\}$ be a fundamental system of neighborhoods of $y$.  According to the claim, we may pick for each $U$ an open set $I_U \subset [0,1]$ such that the restriction functor $\scrC\big([0,1] \times U\big) \to \scrC(I_U \times U)$ is an equivalence of categories.  We may choose the $I_U$ in such a way that the open sets $I_U \times U \subset [0,1] \times Y$ form a fundamental system of neighborhoods of $(t,y)$.  It follows that the natural restriction functor on stalks $(q_* \scrC)_y \to \scrC_{(t,y)}$ is an equivalence.

Now let $x$ be a point in $X$.  We have natural equivalences
$$
\begin{array}{ccccc}
\left(f^* q_* \scrC\right)_x & \cong & \left(q_* \scrC\right)_{f(x)} & \cong & \scrC_{\left( t,f\left(x\right)\right)}\\
\left(p_* g^* \scrC\right)_x & \cong & \left(g^* \scrC\right)_{(t,x)} & \cong & \scrC_{\left(t,f\left(x\right)\right)}
\end{array}
$$

This completes the proof.
\end{proof}

\begin{theorem}[Homotopy invariance]
\label{lcshotinv}
Let $X$ be a topological space, and let $\pi$ denote the projection map $[0,1] \times X \to X$.  Then $\pi_*$
and $\pi^*$ are inverse equivalences between the 2-category of locally constant stacks on $X$, and the 2-category of locally constant stacks on $[0,1] \times X$.
\end{theorem}

\begin{proof}
Let $\scrC$ be a locally constant stack on $X$ and let $\scrD$ be a locally constant stack on $I \times X$.

Let $x \in X$, let $i_x$ denote the inclusion map $\{x\} \hookrightarrow X$, let $j_x$ denote the inclusion map $[0,1] \cong [0,1] \times{x} \hookrightarrow [0,1] \times X$, and let $p$ denote the map $[0,1] \to \{x\}$.  By proposition \ref{basechangelc}, the natural map $(\pi_* \pi^* \scrC)_x = i_x^* \pi_* \pi^* \scrC \to p_* j_x^* \pi^* \scrC$ is an equivalence.  But $p_* j_x^* \pi^* \scrC \cong p_* p^* i_x^* \scrC = p_* p^* (\scrC_x)$.  The natural map $\scrC_x \to (\pi_* \pi^* \scrC)_x \cong p_* p^* (\scrC_x)$ coincides with the adjunction map $\scrC_x \to p_* p^* (\scrC_x)$.  Since $p^* \scrC_x$ is constant, $\scrC_x \to p_* p^* (\scrC_x)$ is an equivalence by proposition \ref{constantstacks}.  It follows that $\scrC \to \pi_* \pi^* \scrC$ is an equivalence of stacks.

Now let $(t,x) \in [0,1] \times X$.  There is an equivalence $(\pi^* \pi_* \scrD)_{(t,x)} \cong (\pi_* \scrD)_x$.  Once again proposition \ref{basechangelc} provides an equivalence $(\pi_* \scrD)_x \cong p_* j_x^* \scrD = j_x^* \scrD\big( [0,1]\big)$.  The locally constant stack $j_x^* \scrD$ is constant on $[0,1]$, so $j_x^* \scrD\big([0,1]\big) \cong \scrD_{(t,x)}$ by proposition \ref{constantstacks}.  It follows that $\pi^* \pi_* \scrD \to \scrD$ is an equivalence of stacks, completing the proof.

\end{proof}

\begin{corollary}
\label{corlcshotinv}
Let $X$ and $Y$ be topological spaces, and let $f:X \to Y$ be a homotopy equivalence.  
\begin{enumerate}
\item The 2-functor $f^*:\St_{lc}(Y) \to \St_{lc}(X)$ is an equivalence of 2-categories.
\item Let $\scrC$ be a locally constant stack on $Y$.  Then the natural functor $\scrC(Y) \to f^* \scrC(X)$ is an equivalence of categories.
\end{enumerate}
\end{corollary}

\begin{proof}
Let $g:Y \to X$ be a homotopy inverse to $f$, and let $H: [0,1] \times X \to X$ be a homotopy between $g \circ f$ and $1_X$.  Let $\pi$ denote the projection map $[0,1] \times X \to X$, and for $t \in [0,1]$ let $i_t$ denote the map $X \to [0,1] \times X: x \mapsto (t,x)$.  Then $i_t^* \cong \pi_*$ by theorem \ref{lcshotinv}.  It follows that $i_0^* \cong i_1^*$, and that $(H \circ i_0)^* \cong (H \circ i_1)^*$.    But $H \circ i_0 = g \circ f$ and $H \circ i_1 = 1_X$, so $f^* \circ g^* \cong 1_{\St_{lc}(X)}$.  Similarly using a homotopy $G:I \times Y \to Y$ we may construct an equivalence $g^* \circ f^* \cong 1_{\St_{lc}(Y)}$.  This proves the first assertion.  

To prove the second assertion, let $*$ denote the trivial category.  By (1), $\Hom_{\St(Y)}(*,\scrC) \cong \Hom_{\St(X)}\big( *, f^*\scrC\big)$.  But $\scrC(Y) \cong \Hom(*,\scrC)$ and $f^*\scrC(X) \cong \Hom(*,f^*\scrC)$.  This completes the proof.
\end{proof}

\begin{theorem}
\label{lcscrit}
Let $X$ be a locally contractible space, and let $\scrC$ be a stack of categories on $X$.  The following are equivalent:
\begin{enumerate}
\item $\scrC$ is locally constant.

\item If $U$ and $V$ are two open subsets of $X$ with $V \subset U$, and the inclusion map $V \hookrightarrow U$ is a homotopy equivalence, then the restriction functor $\scrC(U) \to \scrC(V)$ is an equivalence of categories.

\item If $U$ and $V$ are any two contractible open subsets of $X$, and $V \subset U$, then $\scrC(U) \to \scrC(V)$ is an equivalence of categories.
\item There exists a collection $\{U_i\}$ of contractible open subsets of $X$ such that each point $x \in X$ has a fundamental system of neighborhoods of the form $U_i$, and such that $\scrC(U_i) \to \scrC(U_j)$ is an equivalence of categories whenever $U_j \subset U_i$.
\end{enumerate}

\end{theorem}

\begin{proof}
Suppose $\scrC$ is locally constant, and let $U$ and $V$ be as in condition (2).  Then corollary \ref{corlcshotinv} implies that the restriction functor $\scrC(U) \to \scrC(V)$ is an equivalence of categories, so condition (1) implies condition (2).  Clearly condition (2) implies condition (3), and condition (3) implies condition (4).  Let us show condition (4) implies condition (1).

Suppose $\scrC$ satisfies condition (4).  To show $\scrC$ is locally constant it is enough to show its restriction to each of the distinguished charts in $\{U_i\}$ is constant.  Let $U \subset X$ be such a chart.  Since each point $x \in U$ has a fundamental system of neighborhoods $\{V \}$ from $\{U_i\}$, and since $\scrC(U) \to \scrC(V)$ is an equivalence for each $V$, the map $\scrC(U) \to \scrC_x$ is an equivalence for each $U$.  It follows that the natural map from the constant 
stack $\big(\scrC(U)\big)_U$ to $\scrC\vert_U$ is an equivalence on stalks, and therefore an equivalence.
\end{proof}

\subsection{Direct images and base change}

Let $f$ be a continuous map between locally contractible spaces.  As an application of theorem \ref{lcscrit}, we may give easy proofs of some basic properties of the direct image $f_*$ of locally constant stacks.

\begin{proposition}
Let $X$ and $Y$ be locally contractible spaces.  Let $f:X \to Y$ be a locally trivial fiber bundle, or more generally a Serre fibration.  Let $\scrC$ be a locally constant stack on $X$.  Then $f_* \scrC$ is locally constant on $Y$.
\end{proposition}

\begin{proof}
By theorem \ref{lcscrit}, it suffices to show that the restriction functor $f_* \scrC(U) \to f_* \scrC(V)$, which is equal to the restriction functor $\scrC\big(f^{-1}(U)\big) \to \scrC\big(f^{-1}(V)\big)$, is an equivalence of categories whenever $V \subset U \subset X$ are open sets and $U$ and $V$ are contractible.  Since $f$ is a Serre fibration, the inclusion map $f^{-1}(V) \hookrightarrow f^{-1}(U)$ is a homotopy equivalence, and the proposition follows from corollary \ref{corlcshotinv}.
\end{proof}

\begin{theorem}
\label{serrebasechange}
Let $X$ and $S$ be locally contractible spaces.  Let $p:X \to S$ be a locally trivial fiber bundle, or more generally a Serre fibration.  Let $T$ be another locally contractible space, and let $f:T \to S$ be any continuous function.  Set $Y = X \times_S T$, and let $g:Y \to X$ and $q:Y \to T$ denote the projection maps.
$$\xymatrix{
Y \ar[r]^g \ar[d]_q & X \ar[d]^p\\
T \ar[r]_f & S
}
$$
Let $\scrC$ be a locally constant stack on $X$.  The base-change map $f^* p_* \scrC \to q_* g^* \scrC$ is an equivalence of stacks.
\end{theorem}

\begin{proof}
The statement is local on $T$ and $S$, so we may assume both $T$ and $S$ are contractible.  Since $p$ and $q$ are Serre fibrations the stacks $f^* p_* \scrC$ and $q_* g^* \scrC$ are locally constant, and therefore constant.  To show that the base-change map is an equivalence of stacks it is enough to show that the functor $f^* p_* \scrC(T) \to q_* g^* \scrC(T)$ is an equivalence of categories.  We have $q_* g^* \scrC(T) = g^* \scrC(Y)$, which by corollary $\ref{corlcshotinv}$ is equivalent to $\scrC(X)$.  Furthermore, corollary \ref{corlcshotinv} shows that $f^* p_* \scrC(T) \cong p_*\scrC(S) = \scrC(X)$.  This completes the proof.
\end{proof}

\section{Constructible stacks}
\label{sec-cs}

In this section we introduce constructible stacks.  First we review the topologically stratified spaces of \cite{ih2}.  The definition is inductive: roughly, a stratification $S$ of a space $X$ is a decomposition into pieces called ``strata,'' such that the decomposition looks locally like the cone on a simpler (lower-dimensional) stratified space.  A stack on $X$ is called ``$S$-constructible'' if its restriction to each stratum is locally constant.  This definition is somewhat unwieldy, and we give a more usable criterion in theorem \ref{cscrit}, analogous to theorem \ref{lcscrit} for locally constant stacks: a stack is $S$-constructible if and only if the restriction from a ``conical'' open set to a smaller conical open set is an equivalence of categories.  This criterion is a consequence of a stratified-homotopy invariance statement (corollary \ref{corcshotinv}).

\subsection{Topologically stratified spaces}

\begin{definition}
\label{topstratspace}
Let $X$ be a paracompact Hausdorff space.

A $0$-dimensional topological stratification of $X$ is a homeomorphism between $X$ and a countable discrete set of points.  For $n > 0$, an $n$-dimensional \emph{topological stratification} of $X$ is a filtration
$$\emptyset = X_{-1} \subset X_0 \subset X_1 \subset \ldots \subset X_n = X$$
of $X$ by closed subsets $X_i$, such that for each $i$ and for each point $x \in X_i - X_{i-1}$,
there exists a neighborhood $U$ of $x$, a compact Hausdorff space $L$, an $(n - i - 1)$-dimensional topological stratification
$$\emptyset = L_{-1} \subset L_1 \subset L_2 \subset \ldots \subset L_{n-i-1} = L$$
of $L$, and a homeomorphism $CL \times \bbR^i \cong U$ that takes each $CL_j \times \bbR^i$ homeomorphically to $U \cap X_j$.  
Here $CL = [0,1) \times L / \{0\} \times L$
is the open cone on $L$ if $L$ is non-empty; if $L$ is empty then let $CL$ be a one-point space.

A finite dimensional \emph{topologically stratified space} is a pair $(X,S)$ where $X$ is a paracompact Hausdorff space and $S$ is an $n$-dimensional topological stratification of $X$ for some $n$.
\end{definition}

Let $(X,S)$ be a topologically stratified space with filtration
$$ \emptyset = X_{-1} \subset X_0 \subset X_1 \subset \ldots \subset X_n = X$$
Note the following immediate consequences of the definition:
\begin{enumerate}
\item If $X_i - X_{i-1}$ is not empty, then it is an $i$-dimensional topological manifold.

\item If $U \subset X$ is open then the filtration $U_{-1} \subset U_0 \subset U_1 \subset \ldots$
of $U$, where $U_i = U \cap X_i$, is a topological stratification.
\end{enumerate}

We will call the connected components of $X_i - X_{i-1}$, or unions of them, $i$-dimensional \emph{strata}.  
We will call the neighborhoods $U$ homeomorphic to cones ``conical neighborhoods'':

\begin{definition}
Let $(X,S)$ be an $n$-dimensional topologically stratified space.  An open set $U \subset X$ is called a \emph{conical open subset} of $X$ if $U$ is homeomorphic to $CL \times \bbR^i$ for some $L$ as in definition \ref{topstratspace}.
\end{definition}

\begin{remark}
By definition, every point in a topologically stratified space has an conical neighborhood $CL \times \bbR^k$.  One of the quirks of topological stratifications (as opposed to e.g. Whitney stratifications) is that the space $L$ is not uniquely determined up to homeomorphism: there even exist non-homeomorphic manifolds $L_1$ and $L_2$ such that $CL_1 \cong CL_2$ (see \cite{milnor}).
\end{remark}

The following definition, from \cite{ih2}, is what is usually meant by ``stratified map.''

\begin{definition}
\label{def-stratifiedmaps}
Let $(X,S)$ and $(Y,T)$ be topologically stratified spaces.  A continuous map $f:X \to Y$ is \emph{stratified} if it satisfies the following two conditions:
\begin{enumerate}
\item  For any connected component $C$ of any stratum $Y_k - Y_{k-1}$, the set $f^{-1}(C)$
is a union of connected components of strata of $X$.
\item  For each point $y \in Y_i - Y_{i-1}$ there exists a neighborhood $U$ of $x$ in $Y_i$, a topologically stratified space 
$$F = F_k \supset F_{k-1} \supset \ldots \supset F_{-1} = \emptyset$$
and a filtration-preserving homeomorphism 
$$F \times U \cong f^{-1}(U)$$
that commutes with the projection to $U$.
\end{enumerate}
\end{definition}

We need a much broader definition:

\begin{definition}
\label{def-loosestratifiedmaps}
Let $(X,S)$ and $(Y,T)$ be topologically stratified spaces.  A continuous map $f:X \to Y$ is called \emph{stratum-preserving} if for each $k$, and each connected component $Z \subset X_k - X_{k-1}$, the image $f(Z)$ is contained in $Y_\ell - Y_{\ell -1}$ for some $\ell$.
\end{definition}

\begin{definition}
Let $(X,S)$ and $(Y,T)$ be topologically stratified spaces, and let $f$ and $g$ be two stratum-preserving maps from $X$ to $Y$.  We say $f$ and $g$ are \emph{homotopic relative to the stratifications} if there exists a homotopy $H:[0,1] \times X \to Y$ between $f$ and $g$ such that the map $H(t,-):X \to Y$ is stratum-preserving for every $t \in [0,1]$.
\end{definition}

A slightly irritating feature of this definition is that the space $[0,1] \times X$ cannot be stratified without
treating the boundary components $\{0\} \times X$ and $\{1\} \times X$ differently.  We may take care of this by using the open interval:
if $(X,S)$ and $(Y,T)$ are topologically stratified spaces, then we may endow $(0,1) \times X$ with a topological stratification by setting $\big((0,1) \times X\big)_i = (0,1) \times X_{i-1}$.  Note the following
\begin{enumerate}
\item Let $H:[0,1] \times X \to Y$ be a stratified homotopy.  The restriction of this map to $(0,1) \times X$
is stratum-preserving.
\item Let $f$ and $g$ be two stratum-preserving maps.  Then $f$ and $g$ are homotopic relative to the stratifications if and only if there exists a stratum-preserving map $H:(0,1) \times X \to Y$ such that $f(-) = H(t_0,-)$ and $g(-) = H(t_1,-)$ for some $t_0,t_1 \in (0,1)$.
\end{enumerate}

\begin{definition}
Let $(X,S)$ and $(Y,T)$ be topologically stratified spaces.  Let $f:X \to Y$ be a stratum-preserving map.  Call $f$ a \emph{stratified homotopy equivalence} if there is a stratum-preserving map $Y \to X$ such that the composition $g \circ f$ is stratified homotopic to the identity map $1_X$ of $X$, and $f \circ g$ is stratified homotopic to the identity map $1_Y$ of $Y$.
\end{definition}

Note that a ``stratified homotopy equivalence'' $f$ need not be a stratified map in the sense of definition \ref{def-stratifiedmaps}, but only stratum-preserving.

\subsection{Constructible stacks}

\begin{definition}
\label{def-cs}
Let $(X,S)$ be a topologically stratified space and let $\scrC$ be a stack on $X$.  $\scrC$ is called \emph{constructible} with respect to $S$ if, for each $k$,  $i_k^* \scrC$ is locally constant on $X_k- X_{k-1}$, where $i_k:X_k - X_{k-1} \hookrightarrow X$ denotes the inclusion of the $k$-dimensional stratum into $X$.

Let $\St_S(X)$ denote the full subcategory of the 2-category $\St(X)$ of stacks on $X$ whose objects are the $S$-constructible stacks.
\end{definition}

The pullback of a constructible stack is constructible:

\begin{proposition}
Let $(X,S)$ and $(Y,T)$ be two topologically stratified spaces.  Let $f:X \to Y$ be a stratum-preserving map.
If $\scrC$ is a $T$-constructible stack on $Y$, then $f^*\scrC$ is $S$-constructible on $X$.
\end{proposition}

\begin{proof}
We have to show that $f^* \scrC$ is locally constant on $X_k - X_{k-1}$.  It is enough to show it is locally constant on each connected component.  Let $C$ be a component of $X_k - X_{k-1}$, and let $i:C \to X$ be the inclusion.  Then $i^*f^* \scrC \cong (f \circ i)^* \scrC$.  But $f \circ i: C \to Y$ factors
through $j:Y_{\ell} - Y_{\ell -1} \to Y$ for some $\ell$, so $i^* f^* \scrC$ is obtained from pulling back
$j^* \scrC$ on $Y_{\ell} - Y_{\ell -1}$ to $C$.  By proposition \ref{prop-pullpushlc}, this is locally constant on $C$.
\end{proof}

\begin{proposition}
\label{basechangecs}
Let $(X,S)$ be a topologically stratified space, and let $C$ be a connected stratum of $X$.  Let $i:C \hookrightarrow X$ denote the inclusion map.  Let $p:(0,1) \times C \to C$ and $q:(0,1) \times X \to X$ denote the projection maps, and let $j$ denote the inclusion map $(id,i):(0,1) \times C \hookrightarrow X$.
$$\xymatrix{
(0,1) \times C \ar[r]^j \ar[d]_p & (0,1) \times X \ar[d]^q \\
C \ar[r]_i & X}$$
Endow $(0,1) \times X$ with a topological stratification by setting $\big( (0,1) \times X \big)_k = (0,1) \times X_{k-1}$.  Let $\scrC$ be a stack on $(0,1) \times X$ constructible with respect to this stratification.  Then the base-change map $i^* q_*\scrC \to p_* j^*\scrC$ is an equivalence of stacks.
\end{proposition}

\begin{proof}
Let $x$ be a point in $X$ and let $t$ be a point in $(0,1)$.  As in the proof of proposition \ref{basechangelc}, we may show that the natural map $(q_* \scrC)_x \to \scrC_{(t,x)}$ is an equivalence of categories.  If $x$ lies in the stratum $C$, then we have equivalences of categories:
$$
\begin{array}{rcccl}
(i^* q_* \scrC)_x & \cong & (q_* \scrC)_x & \cong & \scrC_{(t,x)} \\
(p_* j^* \scrC)_x & \cong & (j^* \scrC)_{(t,x)} & \cong & \scrC_{(t,x)}
\end{array}
$$
The base-change map commutes with these, proving the proposition.
\end{proof}

\begin{theorem}[Homotopy invariance]
\label{lem-hotinvconstr}
Let $(X,S)$ be a topologically stratified space.  Endow $(0,1) \times X$ with a topological stratification $T$ by setting $\big( (0,1) \times X \big)_i = (0,1) \times X_{i-1}$.  Let $\pi$ be the stratified projection map $(0,1) \times X \to X$.  The adjoint 2-functors $\pi_*$ and $\pi^*$ induce an equivalence between the 2-category of $S$-constructible stacks on $X$ and the 2-category of $T$-constructible stacks on $(0,1) \times X$.
\end{theorem} 

\begin{proof}
We have to show the maps $\scrC \to \pi_* \pi^* \scrC$ and $\pi^* \pi_* \scrD \to \scrD$ are equivalences, where $\scrC$ is a constructible stack on $X$ and $\scrD$ is a constructible stack on $(0,1) \times X$.  

For each $k$, let $i_k$ denote the inclusion map $X_k - X_{k-1} \hookrightarrow X$.  To prove that $\scrC \to \pi_* \pi^* \scrC$ is an equivalence it suffices to show that $i_k^* \scrC \to i_k^* \pi_* \pi^* \scrC$ is an equivalence for each $k$.  Let $j_k$ denote the inclusion $(0,1) \times (X_k - X_{k-1}) \hookrightarrow (0,1) \times X$, and let $p_k$ denote the projection map $(0,1) \times (X_k - X_{k-1}) \to X_k - X_{k-1}$.
$$\xymatrix{
(0,1) \times \big(X_k - X_{k-1}\big) \ar[r]^{\qquad j_k} \ar[d]_{p_k} & (0,1) \times X \ar[d]^\pi \\
X_k - X_{k-1} \ar[r]^{\qquad i_k} & X}$$
By proposition \ref{basechangecs}, we have an equivalence
$i_k^* \pi_* \pi^* \scrC \cong p_{k*} p_k^* i_k^* \scrC$.  
Since $i_k^*\scrC$ is locally constant on $X_k - X_{k-1}$, the map $i_k^* \scrC \to p_{k*} p_k^* i^* \scrC$ is an equivalence by theorem \ref{lcshotinv}.

To show $\pi^* \pi_* \scrD \to \scrD$ is an equivalence, it is enough to show that for each $k$ the map $j_k^* \pi^* \pi_* \scrD \to \scrD$ is an equivalence.  By proposition \ref{basechangecs} we have $j_k^* \pi^* \pi_* \scrD \cong p_k^* i_k ^* \pi_* \scrD \cong p_k^* p_{k*} j_k^* \scrD$, and since $j_k^* \scrD$ is locally constant the map $p_k^* p_{k*} j_k^* \scrD \to j_k^* \scrD$ is an equivalence of stacks by theorem \ref{lcshotinv}.
\end{proof}

\begin{corollary}
\label{corcshotinv}
Let $(X,S)$ and $(Y,T)$ be topologically stratified spaces, and let $f:X \to Y$ be a stratified homotopy equivalence.  
\begin{enumerate}
\item The 2-functor $f^*:\St_T(Y) \to \St_S(X)$ is an equivalence of 2-categories.
\item Let $\scrC$ be an $S$-constructible stack on $X$.  The functor $\scrC(X) \to f^* \scrC(Y)$ is an equivalence of categories.
\end{enumerate}
\end{corollary}

\begin{proof}
A proof identical to the one of corollary \ref{corlcshotinv} gives both statements.
\end{proof}

\begin{theorem}
\label{cscrit}
Let $(X,S)$ be a topologically stratified space and let $\scrC$ be a stack on $X$.  
The following are equivalent:
\begin{enumerate}

\item $\scrC$ is constructible with respect to the stratification.

\item If $U$ and $V$ are two open subsets of $X$ with $V \subset U$, and if the inclusion map $V \hookrightarrow U$ is a stratified homotopy equivalence, then the restriction functor $\scrC(U) \to \scrC(V)$ is an equivalence of categories.

\item Whenever $U$ and $V$ are conical open subsets of $X$ such that $V \subset U$ and the inclusion map $V \to U$ is a stratified homotopy equivalence, the restriction functor $\scrC(U) \to \scrC(V)$
is an equivalence of categories.

\end{enumerate}
If $\scrC$ satisfies these conditions then the natural functor $\scrC(U) \to \scrC_x$ is an equivalence
of categories whenever $U$ is a conical open neighborhood of $x$.
\end{theorem}

\begin{proof}
Suppose $\scrC$ is constructible, and let $U$ and $V$ be as in (2).  Then $\scrC(U) \to \scrC(V)$ is an equivalence by corollary \ref{corcshotinv}, so (1) implies (2).  Clearly (2) implies (3).  

Suppose now that $\scrC$ satisfies the third condition.  Let $Y = X_k - X_{k-1}$ be a stratum, and let $i:Y \hookrightarrow X$ denote the inclusion map.  Let $\{U\}$ be a collection of conical open sets in $X$ that cover $Y$, and such that each $U \cap Y$ is closed in $U$.  To show that $i^* \scrC$ is locally constant on $Y$ it is enough to show that $j^*(\scrC\vert_U)$ is constant on $Y \cap U$, where $j$ denotes the inclusion $Y \cap U \to U$.  

For each $\epsilon > 0$ let $C_\epsilon L \subset CL$ denote the set $[0,\epsilon) \times L / \{0\} \times L$, and let $B_\epsilon(v)$ denote the ball of radius $\epsilon$ around $v \in \bbR^k$.  Let $\{U_i\}$
be the collection of open subsets of $X$ of the form $C_\epsilon L \times B_\delta(v)$ under the homeomorphism $U \cong CL \times \bbR^k$.  Whenever $U_i$ and $U_j$ are of this form and $U_j \subset U_i$, it is easy to directly construct a stratified homotopy inverse to the inclusion map $U_j\hookrightarrow U_i$; thus by assumption the restriction functor $\scrC(U_i) \to \scrC(U_j)$ is an equivalence.  For each $y \in Y \cap U$ the $U_i$ containing $y$ form a fundamental system of neighborhoods of $y$, so the functor $\scrC(U) \to \scrC_y \cong \big(j^* \scrC\big)_y$ is an equivalence.  It follows that $j^*\scrC$ is equivalent to the constant sheaf on $Y \cap U$ with fiber $\scrC(U)$.  This completes the proof.
\end{proof}

\subsection{Direct images}

\begin{proposition}
Let $(X,S)$ and $(Y,T)$ be topologically stratified spaces, and let $f:X \to Y$ be a stratified map (see definition \ref{def-stratifiedmaps}).  Let $\scrC$ be an $S$-constructible stack on $X$.  Then $f_* \scrC$ is $T$-constructible on $Y$.
\end{proposition}

\begin{proof}
Let $y$ be a point of $Y$.  Let $U \cong \bbR^k \times CL$ be a conical neighborhood of $y$, and let $V \subset U$ be a smaller conical neighborhood such that the inclusion map $V \hookrightarrow U$ is a stratified homotopy equivalence.  We may assume $U$ is small enough so that there exists a topologically stratified space $F$ and a stratum-preserving homeomorphism $f^{-1}(U) \cong F \times U$ that commutes with the projection to $U$.  Then the inclusion map $f^{-1}(V) \hookrightarrow f^{-1}(U)$ is a stratified homotopy equivalence: if $\phi:U \to V$ is a homotopy inverse, then a homotopy inverse to $f^{-1}(V) \hookrightarrow f^{-1}(U)$ is given by $(id,\phi):F \times U \to F \times V$.  By proposition \ref{corcshotinv} and theorem \ref{cscrit} it follows that $f_* \scrC$ is constructible.
\end{proof}

\section{Example: the stack of perverse sheaves}

Let $(X,S)$ be a topologically stratified space.  Let $D^b_S(X)$ denote the bounded constructible derived category of $(X,S)$.  $D^b_S(X)$ is the full subcategory of the bounded derived category of sheaves of abelian groups on $X$ whose objects are the cohomologically constructible complexes of sheaves on $X$; that is, the complexes whose cohomology sheaves are constructible with respect to the stratification of $X$.  For details and references see \cite{ih2}.

We note the following:

\begin{lemma}
\label{dhotinv}
Let $(X,S)$ be a topologically stratified space, let $T$ be the induced stratification on $(0,1) \times X$, and let $\pi:(0,1) \times X \to X$ denote the projection map.  The pullback functor $\pi^*:D^b_S(X) \to D^b_T\big( (0,1) \times X \big)$ is an equivalence of categories.
\end{lemma}

\begin{proof}
A constructible sheaf $F$ on a topologically stratified space $U$ of the form $U = \bbR^k \times CL$
has the property that $H^i(U;F) = 0$ for $i>0$.  We may use this to show that $R^i \pi_* F$ vanishes for $F$ constructible on $(0,1) \times X$ and $i >0$.  Indeed, $R^i \pi_* F$ is the sheafification of the presheaf $U \mapsto H^i\big( (0,1) \times U; F)\big)$ and since every point of $X$ has a fundamental system of neighborhoods of the form $\bbR^k \times CL$ the stalks of this presheaf vanish; it follows that $R^i \pi_* F$ vanishes.  Thus $F \to \bR \pi_* \pi^* F$ is a quasi-isomorphism for every sheaf $F$ on $X$, and $\pi^* \bR \pi_* F \to F$ is a quasi-isomorphism for every constructible sheaf on $(0,1) \times X$, completing the proof.  
\end{proof}

If $C$ is a connected stratum of $X$ let $i_C$ denote the inclusion map $i_C:C \hookrightarrow X$.  Let $D^b_{lc}(C)$ denote the subcategory of $D^b(C)$ whose objects are the complexes with locally constant cohomology sheaves.  Recall the four functors
$\bR i_{C,*}, i_{C,!}:D^b_{lc}(C) \to D^b_S(X)$ and $\bR i_C^!, i_C^*:D^b_S(X) \to D^b_{lc}(C)$,  and recall the following definition from \cite{bbd}:

\begin{definition}
Let $(X,S)$ be a topologically stratified space, and let $p:C \mapsto p(C)$ be any function from connected strata of $(X,S)$ to $\bbZ$.  For each connected stratum $C$, let $i_C$ denote the inclusion $C \hookrightarrow X$.  A \emph{perverse sheaf of perversity $p$} on $X$, constructible with respect to $S$, is a complex $K \in D^b_S(X)$ such that
\begin{enumerate}
\item The cohomology sheaves of $i_C^* K \in D^b(C)$ vanish above degree $p(C)$ for each $C$.
\item The cohomology sheaves of $\bR i_C^! K \in D^b(C)$ vanish below degree $p(C)$ for each $C$.
\end{enumerate}

Let $\bP(X,S,p)$ denote the full subcategory of $D^b_S(X)$  whose objects are the perverse sheaves of perversity $p$.
\end{definition}

Every open set $U \subset X$ inherits a stratification from $X$, and we may form the category $D^b_S(U)$.  This defines a prestack on $X$: there is a restriction functor $D^b_S(U) \to D^b_S(V)$ defined in the obvious way whenever $V \subset U$ are open sets in $X$.  It is easy to see that if $P$ is a perverse sheaf on $U$ then its restriction to $V$ is also a perverse sheaf.  We obtain a prestack $U\mapsto \bP(U,S,p)$.  Write $\scrP_{X,S,p}$ for this prestack.  The following theorem is a result of \cite{bbd}:

\begin{theorem}
Let $X$ be a topologically stratified space with stratification $S$.  Let $p$ be any function from connected strata of $X$ to integers.  The prestack $\scrP_{X,S,p}$ is a stack.
\end{theorem}

We may easily prove, using the criterion in theorem \ref{cscrit}:

\begin{theorem}
\label{pshconstructible}
Let $(X,S)$ be a topologically stratified space.  Let $p$ be any function from connected strata of $X$ to integers.  The stack $\scrP_{X,S,p}$ is constructible.
\end{theorem}

\begin{proof}
Let $U$ and $V$ be open sets in $X$, and suppose $V \subset U$ and that the inclusion map $V \hookrightarrow U$ is a stratified homotopy equivalence.  By lemma \ref{dhotinv}, the restriction map $D^b_S(U) \to D^b_S(V)$ is an equivalence of categories.  It follows that $\scrP(U) \to \scrP(V)$ is also an equivalence.  Thus $\scrP$ is constructible by theorem \ref{cscrit}.
\end{proof}

\section{The fundamental 2-groupoid and 2-monodromy}

In this section we review the unstratified version of our main theorem \ref{maintheorem}: we introduce the fundamental 2-groupoid $\pi_{\leq 2}(X)$ of a space $X$ and prove that the 2-category of locally constant stacks
$\St_{lc}(X)$ is equivalent to the 2-category of $\Cat$-valued functors on $\pi_{\leq 2}(X)$.  Let us call the latter objects ``2-monodromy functors,'' and write $\twomon(X)$ for the 2-category of 2-monodromy functors $F:\pi_{\leq 2}(X) \to \Cat$.  We define a 2-functor
$$\sfN:\twomon(X) \to \St_{lc}(X)$$
and prove that it is essentially fully faithful and essentially surjective.  The most important ingredient is an analog for $\pi_{\leq 2}(X)$ of the classical van Kampen theorem; this is theorem \ref{vankampen}.  The results of this section are essentially contained in \cite{pursuingstacks} and \cite{polesellowaschkies}.

\subsection{The fundamental 2-groupoid}
\label{fun2gpd}

Let $X$ be a compactly generated Hausdorff space, and let $x$ and $y$ be two points of $X$.  A \emph{Moore path} from $x$ to $y$ is a pair $(\lambda,\gamma)$ where $\lambda$ is a nonnegative real number and $\gamma:[0,\lambda] \to X$ is a path with $\gamma(0) = x$ and $\gamma(\lambda) = y$.
Let us write $P(x,y)$ for the space of Moore paths from $x$ to $y$, given the compact-open topology.  We have a concatenation map $P(y,z) \times P(x,y) \to P(x,z)$ defined by the formula 
$$
(\lambda,\gamma) \cdot (\kappa,\beta) = (\lambda + \kappa, \alpha) \text{ where } 
\alpha(t) = \bigg\{
\begin{array}{cl}
\beta(t) & \text{ if }t \leq \kappa \\
\gamma(t-\kappa) & \text{ if }t \geq \kappa
\end{array}
$$
If we give the product $P(y,z) \times P(x,y)$ the Kelly topology (the categorical product in the category of compactly generated Hausdorff spaces), this concatenation map is continuous.  It is strictly associative and the constant paths from $[0,0]$ are strict units.

\begin{definition}
Let $\pi_{\leq 2}(X)$ denote the 2-category whose objects are points of $X$, and whose hom categories $\Hom(x,y)$ are the fundamental groupoids of the spaces $P(x,y)$.  (The discussion above shows that this is a strict 2-category.)
\end{definition}

\begin{remark}
\label{moorehot}
The 2-morphisms in $\pi_{\leq 2}(X)$ are technically equivalence classes of paths $[0,1] \to P(x,y)$.  A path $[0,1] \to P(x,y)$ between $\alpha$ and $\beta$ is given by a pair $(b,H)$ where $b$ is a map $[0,1] \to \bbR_{\geq 0}$, and $H$ is a map from the closed region in $[0,1] \times \bbR_{\geq 0}$ under the graph of $b$:
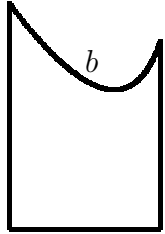
\begin{center}
\setlength{\unitlength}{.1cm}
\begin{picture}(20,30)
\linethickness{.5mm}

\put(0,0){\line(1,0){20}}
\put(0,0){\line(0,1){30}}
\put(20,0){\line(0,1){25}}
\put(20,25){\qbezier(0,0)(-5,-15)(-20,5)}

\put(10,21){$b$}

\end{picture}
\end{center}
$H$ is required to take the top curve to $y$, the bottom curve to $x$, and to map the left and right intervals into $X$ by $\alpha$ and $\beta$.  It is inconvenient and unnecessary to keep track of the function $b$: there is a reparameterization map from Moore paths to ordinary (length 1) paths which takes $(\lambda, \gamma)$ to the path $t \mapsto \gamma(\lambda \cdot t)$.  This map is a homotopy equivalence, so it induces an equivalence of fundamental groupoids.  Thus, 2-morphisms from $\alpha$ to $\beta$ may be represented by homotopy classes of maps $H:[0,1] \times [0,1] \to X$ with the properties
\begin{enumerate}
\item $H(0,u) = \alpha (s \cdot u)$, where $s$ is the length of the path $\alpha$.
\item $H(1,u) = \beta( t \cdot u)$, where $t$ is the length of the path $\beta$.
\item $H(u,0) = x$ and $H(u,1) = y$.
\end{enumerate}
\end{remark}

\subsection{Two-monodromy and locally constant stacks}

\begin{definition}
Let $X$ be a compactly generated Hausdorff space.  Let $\twomon(X)$ denote the 2-category of 2-functors from $\pi_{\leq 2}(X)$ to the 2-category of categories:
$$\twomon(X) := \twofunct\big(\pi_{\leq 2}(X),\Cat\big)$$
\end{definition}

Let $U \subset X$ be an open set.  The inclusion morphism $U \hookrightarrow X$ induces a strict 2-functor $\pi_{\leq 2}(U) \to \pi_{\leq 2}(X)$; let $j_U$ denote this 2-functor.  If $F: \pi_{\leq 2}(X) \to \Cat$ is a 2-monodromy functor on $X$ set $F \vert_U := F \circ j_U$.

\begin{definition}
\label{defsfN}
Let $X$ be a compactly generated Hausdorff space.  Let $\sfN:\twomon(X) \to \Prest(X)$ denote the 2-functor which assigns to a 2-monodromy functor $F:\pi_{\leq 2}(X) \to \Cat$ the prestack
$$\sfN F: U \mapsto \wtwolim{\pi_{\leq 2}(U)}{F\vert_U}$$
\end{definition}

Our goal is to prove that when $X$ is locally contractible $\sfN$ gives an equivalence of 2-categories between $\twomon(X)$ and $\St_{lc}(X)$; this is theorem \ref{twomonod}.

\subsection{A van Kampen theorem for the fundamental 2-groupoid}
\label{vkfun2gpd}

Let $X$ be a compactly generated Hausdorff space.  Let $\{U_i\}_{i \in I}$ be a $d$-cover of $X$.  (By this we just mean that $\{U_i\}_{i \in I}$ is an open cover of $X$ closed under finite intersections; then $I$ is partially ordered by inclusion.  See appendix A.)  An ideal van Kampen theorem would state that the 2-category $\pi_{\leq 2}(X)$ is the direct limit (or ``direct 3-limit'') of the 2-categories $\pi_{\leq 2}(U_i)$.  We do not wish to develop the relevant definitions here.  Instead, we will relate the 2-category $\pi_{\leq 2}(X)$ to the 2-categories $\pi_{\leq 2}(U_i)$ by studying 2-monodromy functors.  We will define a 2-category $\twomon\big(\{U_i\}_{i \in I}\big)$ of ``2-monodromy functors on the $d$-cover,'' and our van Kampen theorem will state that this 2-category is equivalent to $\twomon(X)$.

If $U$ is an open subset of $X$, the inclusion morphism $U \hookrightarrow X$ induces a 2-functor $\pi_{\leq 2}(U) \to \pi_{\leq 2}(X)$.  Let us denote by $(-)\vert_U$ the 2-functor $\twomon(X) \to \twomon(U)$ obtained by composing with $\pi_{\leq 2}(U) \to \pi_{\leq 2}(X)$.

\begin{definition}
Let $\{U_i\}_{i \in I}$ be a $d$-cover of $X$.  A \emph{2-monodromy functor} on $\{U_i\}_{i \in I}$ consists of the following data:
\begin{enumerate}
\item[(0)] For each $i \in I$, a 2-monodromy functor $F_i \in \twomon(U_i)$.
\item[(1)] For each $i,j \in I$ with $U_j \subset U_i$, an equivalence of 2-monodromy functors $F_i \vert_{U_j} \stackrel{\sim}{\to} F_j$.
\item[(2)] For each $i,j,k \in I$ with $U_k \subset U_j \subset U_i$, an isomorphism between the composite equivalence $F_i \vert_{U_j} \vert_{U_k} \stackrel{\sim}{\to} F_j \vert_{U_k} \stackrel{\sim}{\to} F_k$ and the equivalence $F_i \vert_{U_k} \stackrel{\sim}{\to} F_k$
\end{enumerate}
such that the following condition holds:
\begin{enumerate}
\item[(3)] For each $i,j,k,\ell \in I$ with $U_\ell \subset U_k \subset U_j \subset U_i$, the tetrahedron commutes:
$$\tetrah{F_i\vert_{U_j}\vert_{U_k}\vert_{U_\ell}}{F_j\vert_{U_k}\vert_{U_\ell}}{F_k\vert_{U_\ell}}{F_\ell}$$
\end{enumerate}
The 2-monodromy functors on $\{U_i\}_{i \in I}$ form the objects of a 2-category in a natural way.
\end{definition}

If $F$ is a 2-monodromy functor on $X$ then we may form a 2-monodromy functor on $\{U_i\}_{i \in I}$ by setting $F_i = F\vert_{U_i}$, and taking all the 1-morphisms and 2-morphisms to be identities.  This defines a 2-functor $\twomon(X) \to \twomon(\{U_i\})$; let us denote it by $\res$.

\begin{theorem}[van Kampen]
\label{vankampen}
Let $X$ be a compactly generated Hausdorff space, and let $\{U_i\}_{i \in I}$ be a $d$-cover of $X$.  The natural 2-functor $\res:\twomon(X) \to \twomon(\{U_i\}_{i \in I})$ is an equivalence of 2-categories.
\end{theorem}

We will prove this in section \ref{vksubsec}.  Let us first use this result to derive our 2-monodromy theorem.

\subsection{The 2-monodromy theorem}

\begin{theorem}
\label{twomonod}
Let $X$ be a compactly generated Hausdorff space, and let $F$ be a 2-monodromy functor on $X$.
The prestack $\sfN F$ is a stack.  Furthermore, if $X$ is locally contractible, the stack $\sfN F$ is locally constant, each stalk category $(\sfN F)_x$ is naturally equivalent to $F(x)$, and the 2-functor $\sfN: \twomon(X) \to \St_{lc}(X)$ is an equivalence of 2-categories.  
\end{theorem}

\begin{proof}
Let $G$ be another 2-monodromy functor on $X$, and let $\sfN(G,F)$ be the prestack
$U \mapsto \Hom_{\twomon(U)}(G\vert_U,F\vert_U)$.  It is useful to show that $\sfN(G,F)$ is a stack; we obtain that $\sfN F = \sfN(*,F)$ is a stack as a special case.

Let $U \subset X$ be an open set, and let $\{U_i\}_{i \in I}$ be a $d$-cover of $U$.  To see that the natural functor 
$$\sfN(G,F)(U) \to \wtwolim{I}{\sfN(G,F)(U_i)}$$
is an equivalence of categories note that $\twolim_I{\sfN(G,F)(U_i)}$ is equivalent to the category of 1-morphisms from $\res(G\vert_U)$ to $\res(F\vert_U)$.  Here $\res(G\vert_U)$ and $\res(F\vert_U)$ denote the 2-monodromy functors on the $d$-cover $\{U_i\}$ induced by the 2-monodromy functors $G\vert_U$ and $F\vert_U$ on $U$.  By theorem \ref{vankampen}, $\res$ induces an equivalence on hom categories.  Thus $\sfN(G,F)$ is a stack.

Let $U$ and $V$ be contractible open subsets of $X$ with $V \subset U$.  Then both $\pi_{\leq 2}(V)$ and $\pi_{\leq 2}(U)$ are trivial, so $\pi_{\leq 2}(V) \to \pi_{\leq 2}(U)$ in an equivalence.  It follows that $\sfN (G,F)(U) \to \sfN (G,F)(V)$ is an equivalence of categories.  If $X$ is locally contractible then by theorem \ref{lcscrit} $\sfN (G,F)$ is locally constant.  In fact if $U$ is contractible and $x \in U$, the triviality of the 2-category $\pi_{\leq 2}(U)$ shows that $\sfN(G,F)(U)$ is naturally equivalent to the category of functors from $G(x)$ to $F(x)$, thus the stalk $\sfN(G,F)_x$ is equivalent to $\Funct(G(x),F(x))$. 

Now suppose $X$ is locally contractible, and let us show that $\sfN: \twomon(X) \to \St_{lc}(X)$ is essentially fully faithful: we have to show that $\Hom_{\twomon(X)}(G,F) \to \Hom_{\St(X)}(\sfN G, \sfN F)$ is an equivalence of categories.  In fact we will show that the morphism of stacks $\sfN(G,F) \to \shHom(\sfN G,\sfN F)$ is an equivalence.  (Here $\shHom(\sfN G,\sfN F)$ is the stack on $X$ that takes an open set $U$ to the category of 2-natural transformations $\shHom(\sfN G(U), \sfN F(U))$.)  It suffices to show that each of the functors $\sfN(G,F)_x \to \shHom(\sfN G,\sfN F)_x$ between stalks is an equivalence of categories; both these categories are naturally equivalent to the category of functors $\Funct\big(G(x),F(x)\big)$.

Finally let us show that $\sfN:\twomon(X) \to \St_{lc}(X)$ is essentially surjective.  For each 1-category $\bC$, if $F$ is the constant $\bC$-valued 2-monodromy functor on $X$ then $\sfN F$ is the constant stack with fiber $\bC$: the obvious map from the constant prestack $\bC_{p;X}$ to $\sfN F$ induces an equivalence on stalks.  Thus every constant stack is in the essential image of $\sfN$.  Let $\scrC$ be a locally constant stack on $X$, and let $\{U_i\}_{i \in I}$ be a $d$-cover of $X$ over which $\scrC$ trivializes.  Then we may form a 2-monodromy functor on the $d$-cover as follows: for each $i\in I$ we may find an $F_i$ (a constant functor) and an equivalence $\sfN F_i \cong \scrC\vert_{U_i}$; then for each $i,j \in I$ we may form the composite equivalence $F_i \vert_{U_j} \cong \scrC \vert_{U_i} \vert_{U_j} = \scrC \vert_{U_j} \cong F_j$; etc.  By theorem \ref{vankampen}, this descends to a 2-monodromy functor $F$ on $X$, and $\sfN F$ is equivalent to $\scrC$.  This completes the proof.
\end{proof}

\subsection{The proof of the van Kampen theorem}
\label{vksubsec}

Before proving theorem \ref{vankampen} let us discuss homotopies in more detail.  We wish to show that any 2-morphism in $\pi_{\leq 2}(X)$ may be factored into smaller 2-morphisms, where ``small'' is interpreted in terms of an open cover of $X$.  

\begin{definition}
Let $X$ be a compactly generated Hausdorff space.  Let $\{U_i\}_{i \in I}$ be a $d$-cover of $X$.  A homotopy $h:[0,1] \times [0,1] \to X$ is \emph{$i$-elementary} if there is a subinterval $[a,b] \subset [0,1]$ such that $h(s,t)$ is independent of $s$ so long as $t \notin [a,b]$, and such that the image of $[0,1] \times [a,b] \subset [0,1] \times [0,1]$ under $h$ is contained in $U_i$.  If a homotopy $h$ is $i$-elementary for some unspecified $i \in I$ then we will simply call $h$ \emph{elementary}.
\end{definition}

Let $X$ and $\{U_i\}$ be as in the definition.  Let $x,y \in X$ be points, $\alpha,\beta \in P(x,y)$ be Moore paths, and let $h:[0,1] \times [0,1] \to X$ be a homotopy from $\alpha$ to $\beta$.  (See remark \ref{moorehot}.)  Suppose we have paths $\gamma_0,\gamma_1,\alpha'$ and $\beta'$, and a homotopy $h':\alpha' \to \beta'$, such that $\alpha = \gamma_1 \cdot \alpha' \cdot \gamma_0$, $\beta = \gamma_1 \cdot \beta' \cdot \gamma_0$, and $h = 1_{\gamma_1} \cdot h' \cdot 1_{\gamma_0}$.  

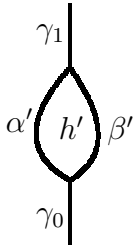
\begin{center}
\setlength{\unitlength}{.05cm}
\begin{picture}(0,70)
\label{bead}
\linethickness{.5mm}

\put(0,21){\qbezier(0,0)(-18,10)(0,30)}
\put(0,21){\qbezier(0,0)(15,9)(0,30)}

\put(0,4){\line(0,1){17}}
\put(0,51){\line(0,1){17}}
\put(-3,32){$h'$}
\put(-17,34){$\alpha'$}
\put(10,32){$\beta'$}

\put(-9,10){$\gamma_0$}
\put(-9,59){$\gamma_1$}

\end{picture}
\end{center}

Then $h$ is an $i$-elementary homotopy if and only if the image of $h'$ lies in $U_i$.  Any $i$-elementary homotopy may be written as $\gamma_1 \cdot h' \cdot \gamma_0$ for some $\gamma_0$, $h'$, $\gamma_1$.  

\begin{proposition}
\label{elhotfactor}
Let $X$ be a compactly generated Hausdorff space, and let $\{U_i\}_{i \in I}$ be a $d$-cover of $X$.
Let $\alpha$ and $\beta$ be two Moore paths from $x$ to $y$, and let $h:[0,1] \times [0,1] \to X$ be a homotopy from $\alpha$ to $\beta$.  Then there is a finite list $\alpha = \alpha_0, \alpha_1, \ldots, \alpha_n = \beta$ of Moore paths from $x$ to $y$, and of homotopies $h_1:\alpha_0 \to \alpha_1$, $h_2:\alpha_1 \to \alpha_2$, \ldots $h_n:\alpha_{n-1} \to \alpha_n$ such that $h$ is homotopic to $h_n \circ h_{n-1} \circ \ldots \circ h_1$, and such that each $h_i$ is elementary.
\end{proposition}

\begin{proof}
Pick a continuous triangulation of $[0,1] \times [0,1]$ with the property that each triangle is mapped by $h$ into one of the $U_i$.  Let $n$ be the number of triangles, and suppose we have constructed an appropriate factorization whenever the square may be triangulated with fewer than $n$ triangles.  Pick
an edge along $\{0\} \times [0,1]$; this edge is incident with a unique triangle $\sigma$, as in the diagram
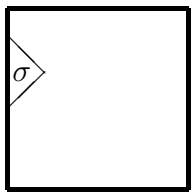
\begin{center}
\setlength{\unitlength}{.12cm}

\begin{picture}(0,30)
\linethickness{.5mm}

\put(-10,5){\line(0,1){20}}
\put(10,5){\line(0,1){20}}
\put(-10,25){\line(1,0){20}}
\put(-10,5){\line(1,0){20}}

\put(-9.5,17){$\sigma$}

\put(-10,22){\line(1,-1){4.1}}
\put(-10,14){\line(1,1){4.1}}

\end{picture}
\end{center}
We may find a homeomorphism $\eta$ between the complement of $\sigma$ in this square with another square such that the composition  
$$[0,1] \times [0,1] \stackrel{\eta}{\cong} \text{closure}\big([0,1] \times [0,1] - \sigma\big) \to X$$
may be triangulated with $n-1$ triangles.  Let us denote this composition by $g$.  On the other hand it is clear how to parameterize the union of $\sigma$ and $\{0\} \times [0,1]$ by an elementary homotopy:
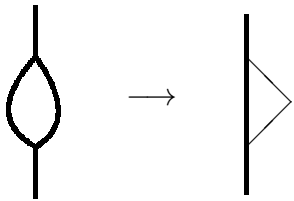
\begin{center}
\setlength{\unitlength}{.04cm}
\begin{picture}(0,70)
\linethickness{.5mm}

\put(-40,21){\qbezier(0,0)(-18,10)(0,30)}
\put(-40,21){\qbezier(0,0)(15,9)(0,30)}

\put(-40,4){\line(0,1){17}}
\put(-40,51){\line(0,1){17}}

\put(30,5){\line(0,1){60}}
\put(30,51){\line(1,-1){15}}
\put(30,21){\line(1,1){15}}

\put(-10,35){$\longrightarrow$}

\end{picture}
\end{center}
Let us write $k:[0,1] \times [0,1] \to X$ for the composition of this parameterization with $h$.  Now $k$ is an elementary homotopy and $g$ may be factored into elementary homotopies by induction.  The mapping cylinders on the homeomorphism $\eta$ and the parameterization of $\sigma \cup \{0\} \times [0,1]$ form a homotopy between $h$ and $g \circ k$.
\end{proof}

We also need a notion of elementary 3-dimensional homotopy.

\begin{definition}
\label{3delhot}
Let $X$ be a compactly generated Hausdorff space, and let $\{U_i\}_{i \in I}$ be a $d$-cover of $X$.  Let $x, y \in X$, $\alpha, \beta \in P(x,y)$, and let $h_0,h_1:[0,1] \times [0,1] \to X$ be homotopies from $\alpha$ to $\beta$.  A homotopy $t \mapsto h_t$ between $h_0$ and $h_1$ is called \emph{$i$-elementary} if there is a closed rectangle $[a,b] \times [c,d] \subset [0,1] \times [0,1]$ such that 
\begin{enumerate}
\item $h_t(u,v)$ is independent of $t$ for $(u,v) \notin [a,b] \times [c,d]$
\item For each $t$, $h_t([a,b] \times [c,d]) \subset U_i$.
\end{enumerate}
\end{definition}

\begin{proposition}
\label{3delhotfactor}
Let $X$, $\{U_i\}_{i \in I}$, $x,y,\alpha,\beta$ be as in definition \ref{3delhot}.  Let $h$ and $g$ be homotopies from $\alpha$ to $\beta$.  Suppose that $h$ and $g$ are homotopic.  Then there is a sequence $h = k_0,k_1,\ldots,k_n = g$ of homotopies from $\alpha$ to $\beta$ such that $k_i$ is homotopic to $k_{i+1}$ via an elementary homotopy.
\end{proposition}

\begin{proof}
Note that the homotopy between $h$ and a factorization $h_m \circ \ldots \circ h_1$ constructed in propostion \ref{elhotfactor} is given by a sequence of elementary 3-dimensional homotopies.  Thus we may assume that $h$ is of the form $h_m \circ \ldots \circ h_1$, where each $h_i$ is an elementary homotopy, and that $g$ is of the form $g_\ell \circ \ldots \circ g_1$ where each $g_i$ is an elementary homotopy.  By induction we may reduce to the case where $m = \ell = 1$ so that $h$ and $g$ are both elementary.  Suppose that $H:[0,1] \times [0,1] \times [0,1]$ is a homotopy between $h$ and $g$.
We may triangulate $[0,1] \times [0,1] \times [0,1]$ in such a way that each simplex $\sigma$ is carried by $H$ into one of the charts $U_i$.  We may use these simplices to factor $H$ just as in proposition \ref{elhotfactor}.
\end{proof}

Now we may prove theorem \ref{vankampen}.  To show that $\res:\twomon(X) \to \twomon\big(\{U_i\}\big)$ is an equivalence of 2-categories it suffices to show that $\res$ is essentially fully faithful and essentially surjective. This is the content of the following three propositions.  

\begin{proposition}
\label{res2full}
Let $X$ be a compactly generated Hausdorff space, and let $\{U_i\}_{i \in I}$ be a $d$-cover of $X$.  Let $F$ and $G$ be two 2-monodromy functors on $X$.  The functor $\Hom(F,G) \to \Hom\big(\res(F),\res(G)\big)$ induced by $\res$ is fully faithful.
\end{proposition}

\begin{proof}
Let $n$ and $m$ be two 2-natural transformations $F \to G$, and let $\phi$ and $\psi$ be two modifications $n \to m$.  Since we have $\phi = \psi$ if and only if $\phi_x = \psi_x$ for each $x \in X$, and since $\phi_x = \res(\phi)_{x,i}$ and $\psi_x = \res(\psi)_{x,i}$ whenever $x \in U_i$, we have $\res(\phi) = \res(\psi)$ if and only if $\phi = \psi$.  This proves that the functor induced by $\res$ is faithful.

Let $n$ and $m$ be as before, and now let $\{\phi_{x,i}\}$ be a 2-morphism between $\res(n)$ and $\res(m)$.  The 2-morphisms $\phi_{x,i}:n(x) \to m(x)$ are necessarily independent of $i$, since whenever $x \in U_j \subset U_i$ we have $\phi_{x,j} \circ 1_{n(x)} = 1_{m(x)} \circ \phi_{x,i}$.  Write $\phi_x = \phi_{x,i}$ for this common value.  Since $\{\phi_{x,i}\}$ is a 2-morphism in $\twomon\big(\{U_i\}_{i \in I}\big)$, every path $\gamma:x \to y$ whose image is contained in one of the $U_i$ induces a commutative diagram
$$
\xymatrix{
n(x) \ar[r]^{n(\gamma)} \ar[d]_{\phi_x} & n(y) \ar[d]^{\phi_y} \\
m(x) \ar[r]_{m(\gamma)} & m(y)
}
$$
It follows that this diagram commutes for every path $\gamma$, since every $\gamma$ may be written as a concatenation $\gamma_N \cdot \ldots \cdot \gamma_1$ of paths $\gamma_k$ with the property that for each $k$ there is an $i$ such that the image of $\gamma_k$ is contained in $U_i$.  Thus, $x \mapsto \phi_x$ is a 2-morphism $n \to m$, and $\res\big( \{\phi_x\}\big) = \{\phi_{x,i}\}$, so the functor induced by $\res$ is full.
\end{proof}

\begin{proposition}
\label{resessfull}
Let $X$, $\{U_i\}_{i \in I}$, $F$, and $G$ be as in proposition \ref{res2full}.  The functor $\Hom(F,G) \to \Hom\big( \res(F),\res(G)\big)$ induced by $\res$ is essentially surjective.
\end{proposition}

\begin{proof}

Let $\{n_i\}$ be a 1-morphism $\res(F) \to \res(G)$.  For each $i$ and each $x$ with $x \in U_i$ we are given a functor $n_i(x):F(x) \to G(x)$, and for each $j$ with $x \in U_j \subset U_i$ we are given a natural isomorphism $\rho_{ij;x}:n_i(x) \stackrel{\sim}{\to} n_j(x)$ which makes certain diagrams commute.  In particular, if we have $x \in U_k \subset U_j \subset U_i$, then $\rho_{jk;x} \circ \rho_{ij;x} = \rho_{ik;x}$.  Let us take
$$n(x) := \varinjlim_{i \in I \, \mid \, U_i \owns x} n_i(x)$$
Since the limit is filtered and each $n_i(x) \to n_j(x)$ is an isomorphism, the limit exists and all the natural maps $n_i(x) \to n(x)$ are isomorphisms. 

To show that $\{n_i\}$ is in the essential image of $\res$ we will extend the assignment $x \mapsto n(x)$
to a 1-morphism $F \to G$.  To do this we need to define an isomorphism $n(\gamma):G(\gamma) \circ n(x) \stackrel{\sim}{\to} n(y) \circ F(\gamma)$ for every path $\gamma$ starting at $x$ and ending at $y$.  In case the image of $\gamma$ is entirely contained in $U_i$ for some $i$, define $n(\gamma)$ to be the composition
$$G(\gamma) \circ n(x) \stackrel{\sim}{\leftarrow} G(\gamma) \circ n_i(x) \stackrel{n_i(\gamma)}{\to} n_i(y) \circ F(\gamma) \stackrel{\sim}{\to} n(y) \circ F(\gamma)$$
By naturality of the morphisms $\rho_{ij;x}$, this map is independent of $i$.  For general $\gamma$ we may find a factorization $\gamma = \gamma_1 \cdot \ldots \cdot \gamma_N$ where each $\gamma_k$ is contained in some $U_\ell$, and define $n(\gamma) = n(\gamma_1)  n(\gamma_2) \ldots n(\gamma_k)$.  

Let $x$ and $y$ be points in $X$, let $\alpha$ and $\beta$ be two paths from $x$ to $y$, and let $h$ be a homotopy from $\alpha$ to $\beta$.  To show that the assignments $x \mapsto n(x)$ and $\gamma \mapsto n(\gamma)$ form a 1-morphism $F \to G$, we have to show that $n(\alpha)$ and $n(\beta)$ make the following square commute:
$$\xymatrix{
{n(y) \circ F(\alpha)} \ar[r]^{n(y)F(h)} 
\ar[d]_{n(\alpha)} & {n(y) \circ F(\beta)} 
\ar[d]^{n(\beta)} \\
{G(\alpha) \circ n(x)}
\ar[r]_{G(h)n(x)} & {G(\beta) \circ n(x)}
}
$$
By proposition \ref{elhotfactor} we may assume $h$ is elementary.  An elementary homotopy may be factored as $h = 1_\gamma \cdot h' \cdot 1_\delta$ where the image of $h'$ lies in $U_i$, so we may as well assume the image of $h$ lies in $U_i$.  In that case the diagram above is equivalent to 
$$\xymatrix{
{n_i(y) \circ F\vert_i(\alpha)} \ar[rr]^{n_i(y) F\vert_i(h)} \ar[d]_{n_i(\alpha)} & & {n_i(y) \circ F\vert_i(\beta)}
\ar[d]^{n_i(\beta)} \\
{G\vert_i(\alpha) \circ n_i(x)} \ar[rr]_{G\vert_i(h) n_i(x)} & & {G\vert_i(\beta) \circ n_i(x)}
}
$$
which commutes by assumption.  (Here $F\vert_i$ and $G\vert_i$ denote the restrictions of $F$ and $G$ to $\pi_{\leq 2}(U_i)$.)  The natural isomorphisms $n_i(x) \to n(x)$ assemble to an isomorphism between $\res(n)$ and $\{n_i\}$, completing the proof.

\end{proof}

\begin{proposition}
\label{resesssurj}
Let $X$ be a compactly generated Hausdorff space, and let $\{U_i\}$ be a $d$-cover of $X$.  The natural 2-functor $\res:\twomon(X) \to \twomon\big(\{U_i\}\big)$ is essentially surjective.
\end{proposition}

\begin{proof}
Let $\{F_i\}$ be an object of $\twomon\big( \{U_i\}_{i \in I}\big)$.  For each point $x \in X$ let $F(x)$ denote the category
$$F(x) := \wtwocolim{i \in I \, \mid \, x \in U_i}{F_i(x)}$$
Since $I$ is filtered and each of the maps $F_i(x) \to F_j(x)$ is an equivalence, the natural map $F_i(x) \to F(x)$ is an equivalence of categories for each $i$.  

We wish to extend the assignment $x \mapsto F(x)$ to a 2-functor $\pi_{\leq 2}(X) \to \Cat$.  Let $\gamma$ be a path between points $x$ and $y$ in $X$.  If the image of $\gamma$ is contained in some $U_i$ then we may form $F(\gamma):F(x) \to F(y)$ by taking the direct limit over $i$ of the functors $c_{i,\gamma}:F(x) \to F(y)$, where $c_{i,\gamma}$ is the composition 
$$F(x) \stackrel{\sim}{\leftarrow} F_i(x) \stackrel{F_i(\gamma)}{\to} F_i(y) \stackrel{\sim}{\to} F(y)$$
Whenever $U_j \subset U_i$ the natural transformation $c_{i,\gamma} \to c_{j,\gamma}$ induced by the commutative square
$$\xymatrix{
F_i(x) \ar[r]^{F_i(\gamma)} \ar[d] & F_i(y) \ar[d] \\
F_j(x) \ar[r]_{F_j(\gamma)} & F_j(y)}
$$
is an isomorphism, and the limit is filtered, so each of the maps $c_{i,\gamma} \to F(\gamma)$ is an isomorphism.

Now for each path $\gamma$ not necessarily contained in one chart,  pick a factorization $\gamma = \gamma_1 \cdot \ldots \cdot \gamma_N$ with the property that for each $\ell$ there is a $k$ such that the image of $\gamma_\ell$ lies in $U_k$.  If $x_\ell$ and $x_{\ell+1}$ denote the endpoints of $\gamma_\ell$, let $F(\gamma):F(x) \to F(y)$ be the functor given by the composition 
$$F(x) = F(x_1) \stackrel{F(\gamma_1)}{\to} F(x_2) \stackrel{F(\gamma_2)}{\to} \ldots \stackrel{F(\gamma_N)}{\to} F(x_{N+1}) = F(y)$$

Suppose $h$ is a homotopy between paths $\alpha$ and $\beta$ with the property that the images of $\alpha$, $\beta$, and $h$ lie in a single chart $U_i$.  Then define a natural transformation $F(h):F(\alpha) \to F(\beta)$ to be the composition
$$F(\alpha) \stackrel{\sim}{\leftarrow} F_i(\alpha)  \stackrel{F_i(h)}{\to} F_i(\beta) \stackrel{\sim}{\to} F(\beta)$$
If $h = 1_{\gamma_1} \cdot h' \cdot 1_{\gamma_0}$ is an elementary homotopy, such that the image of $h'$ lies in some $U_i$, define $F(h) = 1_{F(\gamma_1)} \cdot F(h') \cdot 1_{F(\gamma_0)}$.  If $g$ is an arbitrary homotopy, let $g_n \circ g_{n-1} \cdots \circ g_1$ be a composition of elementary homotopies that is homotopic to $g$, and define $F(g) = F(g_n) \circ \cdots \circ F(g_1)$.  The $g_i$ exist by proposition \ref{elhotfactor}, and the formula for $F(g)$ is independent of the factorization by proposition \ref{3delhotfactor}.

We may extend $F$ to all elementary homotopies, since any elementary homotopy can be written as $1_{\gamma_1} \cdot h' \cdot 1_{\gamma_0}$ where the image of $h'$ lies in some $U_i$; it follows that if $h$ and $g$ are elementary homotopies that are themselves homotopic by an elementary homotopy, then $F(h) = F(g)$.  
By propositions \ref{elhotfactor} and \ref{3delhotfactor} this is well-defined.

The maps $F_i(x) \to F(x)$ assemble to a map $\{F_i\} \to \res(F)$ in $\twomon\big( \{U_i\}\big)$.  As each $F_i(x) \to F(x)$ is an equivalence by construction, this shows that $\{F_i\}$ is equivalent to $\res(F)$, so that $\res$ is essentially surjective.
\end{proof}

\section{Stratified 2-truncations and 2-monodromy}

In this section we develop an abstract version of our main theorem.  We introduce the notion of a \emph{stratified 2-truncation}.  A stratified 2-truncation $\orpi_{\leq 2}$ is a strict functorial assignment from topologically stratified spaces to 2-categories satisfying a few axioms.  We show that these axioms guarantee that the 2-category of 2-functors from $\orpi_{\leq 2}(X,S)$ to $\Cat$ is equivalent to the 2-category of $S$-constructible stacks on $X$.

Let $\Strat$ denote the category of topologically stratified spaces and stratum-preserving maps between them.  We will consider functors from $\Strat$ to the category (that is, 1-category) of 2-categories and strict 2-functors; we will denote the latter category by $\twocat$.  Thus, such a functor $\orpi_{\leq 2}$ consists of
\begin{enumerate}
\item an assignment $(X,S) \mapsto \orpi_{\leq 2}(X,S)$ that takes a topologically stratified space to a 2-category.

\item an assignment $f \mapsto \orpi_{\leq 2}(f)$ that takes a stratum-preserving map $f:X \to Y$ to a strict 2-functor $\orpi_{\leq 2}(f):\orpi_{\leq 2}(X) \to \orpi_{\leq 2}(Y)$.
\end{enumerate}
such that for any pair of composable stratum-preserving maps $X \stackrel{f}{\to} Y \stackrel{g}{\to} Z$, we have $\orpi_{\leq 2}(g \circ f) = \orpi_{\leq 2}(g) \circ \orpi_{\leq 2}(f)$.

A functor $\orpi_{\leq 2}$ is called a \emph{stratified 2-truncation} if it satisfies the four axioms below. Two of these axioms require some more discussion, but we will state them here first somewhat imprecisely:

\begin{definition}
\label{2truncaxioms}
Let $\orpi_{\leq 2}$ be a functor $\Strat \to \twocat$.  We will say that $\orpi_{\leq 2}$ is a \emph{stratified 2-truncation} if it satisfies the following axioms:
\begin{enumerate}
\item[(N)] Normalization.  If $\emptyset$ denotes the empty topologically stratified space, then $\orpi_{\leq 2}(\emptyset)$ is the empty 2-category.
\item[(H)] Homotopy invariance.  For each topologically stratified space $(X,S)$, the 2-functor
$\orpi_{\leq 2}(f):\orpi\big( (0,1) \times X,S' \big) \to \orpi(X,S)$ induced by the projection map $f:(0,1) \times X \to X$ is an equivalence of 2-categories.  Here $S'$ denotes the stratification on $(0,1) \times X$ induced by $S$.
\item[(C)] Cones.  Roughly, for each compact topologically stratified space $L$, $\orpi_{\leq 2}(CL)$ may be identified with the cone on the 2-category $\orpi_{\leq 2}(L)$.  See section \ref{axiomc} below.

\item[(vK)] van Kampen.  Roughly, for every topologically stratified space $X$ and every $d$-cover $\{U_i\}_{i \in I}$ of $X$, the 2-category $\orpi_{\leq 2}(X)$ is naturally equivalent to the direct limit (or ``3-limit'') over $i \in I$ of the 2-categories $\orpi_{\leq 2}(U_i)$.  See section \ref{axiomvk} below.
\end{enumerate}

\end{definition}

\subsection{Cones on 2-categories and axiom (C)}
\label{axiomc}

If $\bC$ is a 2-category, let 
$\big(* \downarrow \bC\big)$ denote the 
2-category whose objects are the objects of $\bC$ together with one new object 
$*$, and where 
the hom categories $\Hom_{*\downarrow \bC}(x,y)$ are as follows:
\begin{enumerate}
\item $\Hom(x, y) = \Hom_{\bC}(x, y)$ if both $x$ and $y$ are in $\bC$. 
\item $\Hom(x, y)$ is the trivial category if $x = *$. 
\item $\Hom(x, y)$ is the empty category if $y = *$ and $x \neq *$.
\end{enumerate}

\begin{definition}
Let $\orpi_{\leq 2}:\Strat \to \twocat$ be a functor satisfying axioms (N) and (H) above.  For each compact topologically stratified space $L$, let us endow $(0,1) \times L$ and $CL$ with the naturally induced topological stratification.  Let us say that
$\orpi_{\leq 2}$ \emph{satisfies axiom (C)} if for each compact topologically stratified space $L$ there is
an equivalence of 2-categories $\orpi_{\leq 2}(CL) \stackrel{\sim}{\to} \big(* \downarrow \orpi_{\leq 2}(L)\big)$ such that
\begin{enumerate}
\item the following square commutes up to equivalence of 2-functors:
$$
\xymatrix{
\orpi_{\leq 2}\big( (0,1) \times L \big) \ar[r] \ar[d] & \orpi_{\leq 2}(CL) \ar[d] \\
\orpi_{\leq 2}(L) \ar[r] & \big({*} \downarrow \orpi_{\leq 2}(L)\big)
}
$$

\item The composition $\orpi_{\leq 2}\big(\{\text{cone point}\}\big) \to \orpi_{\leq 2}(CL) \to \big(* \downarrow \orpi_{\leq 2}(L)\big)$ is equivalent to the natural inclusion 
$\orpi_{\leq 2}\big(\{\text{cone point}\}\big) \cong * \to \big(* \downarrow \orpi_{\leq 2}(L)\big)$
\end{enumerate}
\end{definition}

\subsection{Exit 2-monodromy functors and axiom (vK)}
\label{axiomvk}

Morally, the van Kampen axiom states that $\orpi_{\leq 2}$ preserves direct limits (at least in a diagram of open immersions).  We find it inconvenient to define a direct 3-limit of 2-categories directly; we will instead formulate it in terms of category-valued 2-functors on the 2-categores $\orpi_{\leq 2}(X)$, as in section \ref{vkfun2gpd}.

In this section, fix a functor $\orpi_{\leq 2}:\Strat \to \twocat$.

\begin{definition}
Let $(X,S)$ be a topologically stratified space.  An \emph{exit 2-monodromy functor} on $(X,S)$ with respect to $\orpi_{\leq 2}$ is a 2-functor 
$\orpi_{\leq 2}(X,S) \to \Cat$.  Write $\twoexit(X,S) = \twoexit(X,S;\orpi_{\leq 2})$ for the 2-category of exit 2-monodromy functors on $(X,S)$ with respect to $\orpi_{\leq 2}$.
\end{definition}

\begin{definition}
Let $(X,S)$ be a topologically stratified space.  Let $\{U_i\}_{i \in I}$ be a $d$-cover of $X$.  Endow each $U_i$ with the topological stratification $S_i$ inherited from $S$.  An \emph{exit 2-monodromy functor} on $\{U_i\}_{i \in I}$, with respect to $\orpi_{\leq 2}$ consists of the following data:
\begin{enumerate}
\item[(0)] For each $i \in I $ a 2-monodromy functor $F_i \in \twoexit(U_i, \orpi_{\leq 2})$.
\item[(1)] For each $i,j \in I$ with $U_j \subset U_i$ an equivalence of exit 2-monodromy functors $F_i \vert_{U_j} \stackrel{\sim}{\to} F_j$.
\item[(2)] For each $i,j,k \in I$ with $U_k \subset U_j \subset U_i$ an isomorphism between the composite equivalence $F_i \vert_{U_j} \vert_{U_k} \stackrel{\sim}{\to} F_j \vert_{U_k} \stackrel{\sim}{\to} F_k$ and the equivalence $F_i \vert_{U_k} \stackrel{\sim}{\to} F_k$
\end{enumerate}
such that the following condition holds:
\begin{enumerate}
\item[(3)] For each $i,j,k,\ell \in I$ with $U_\ell \subset U_k \subset U_j \subset U_i$, the tetrahedron commutes:
$$\tetrah{F_i\vert_{U_j}\vert_{U_k}\vert_{U_\ell}}{F_j\vert_{U_k}\vert_{U_\ell}}{F_k\vert_{U_\ell}}{F_\ell}$$
\end{enumerate}
Write $\twoexit\big( \{U_i\}_{i \in I}, \orpi_{\leq 2}\big)$ for the 2-category of exit 2-monodromy functors on $\{U_i\}$.
\end{definition}

Let $X$ be a topologically stratified space, and let $\{U_i\}_{i \in I}$ be a $d$-cover of $X$.  Denote by $\res$ the natural strict 2-functor $\twoexit(X) \to \twoexit\big( \{U_i\}_{i \in I}\big)$.

\begin{definition}
Let $\orpi_{\leq 2}$ be a 2-functor $\Strat \to \twocat$.  We say that $\orpi_{\leq 2}(X)$ \emph{satisfies axiom (vK)} if $\res:\twoexit(X) \to \twoexit\big( \{U_i\}_{i \in I}\big)$ is an equivalence of 2-categories for every topologically stratified space $X$ and every $d$-cover $\{U_i\}_{i \in I}$ of $X$.
\end{definition}

\subsection{The exit 2-monodromy theorem}

In this section fix a stratified 2-truncation $\orpi_{\leq 2}$.

\begin{definition}
Let $(X,S)$ be a topologically stratified space.  For each open set $U \subset X$, let $S_U$ denote the induced stratification of $U$ and let $j_U$ denote the inclusion map $U \hookrightarrow X$.  
Let $\sfN:\twoexit(X,S) \to \Prest(X)$ denote the 2-functor which assigns to an exit 2-monodromy functor $F: \orpi_{\leq 2}(X,S) \to \Cat$ the prestack 
$$\sfN F: U \mapsto \wtwolim{\orpi_{\leq 2}(U,S_U)}{F \circ \orpi_{\leq 2}(j_U)} $$
\end{definition}

We wish to prove that $\sfN$ is an equivalence of $\twoexit(X,S)$ onto the 2-category $\St_S(X)$.

We need a preliminary result about constructible stacks on cones.

\begin{definition}
Let $(L,S)$ be a topologically stratified space.
Let $\big(\Cat \downarrow \St_S(L)\big)$ denote the 2-category whose objects are
triples $(\bC,\scrC,\phi)$, where
\begin{enumerate}
\item $\bC$ is a 1-category.
\item $\scrC$ is a constructible stack on $L$
\item $\phi$ is a 1-morphism $\bC_L \to \scrC$, where $\bC_L$ denotes the constant stack on $L$.
\end{enumerate}
\end{definition}

If $(L,S)$ is a compact topologically stratified space, let $S'$ denote the induced stratification on $(0,1) \times L$ and $S''$ the induced stratification on $CL$.  There is a 2-functor
$$\St_{S''}(CL) \to \left(\Cat \downarrow \St_{S'}\left( \left(0,1\right) \times L \right) \right) \cong \left(\Cat \downarrow \St_S\left(L\right)\right)$$
which associates to a stack $\scrC$ the triple $\big(\scrC(X), \scrC\vert_{(0,1) \times L}, \phi\big)$, where $\phi$ is the evident restriction map.

\begin{definition}
Let $(L,S)$ be a topologically stratified space.
Let $\big(\Cat \downarrow \twoexit(L)\big)$ denote the 2-category whose objects are
triples $(\bC, F, \phi)$ where 
\begin{enumerate}
\item $\bC$ is a 1-category.
\item $F$ is an exit 2-monodromy functor on $L$.
\item $\phi$ is a 1-morphism from the constant $\bC$-valued functor to $F$.
\end{enumerate}
\end{definition}

Note that the equivalence 
$$\orpi_{\leq 2}(CL) \stackrel{\sim}{\to} \big(* \downarrow \orpi_{\leq 2}(L)\big)$$
gives an equivalence 
$$\left(\Cat \downarrow \twoexit(L)\right) \stackrel{\sim}{\to} \twoexit(CL)$$.

\begin{proposition}
\label{twomonforcones}
Let $L$ be a compact topologically stratified space, and let $CL$ be the open cone on $L$.
The 2-functor $\St_{S''}(CL) \to \big(\Cat \downarrow \St_S(L)\big)$ is an equivalence of 2-categories.  Furthermore, the square
$$
\xymatrix{
\twoexit(CL,S'') \ar[r]^{\quad \sfN} \ar[d] & \St_{S''}(CL) \ar[d] \\
\big(\Cat \downarrow \twoexit(L,S)\big) \ar[r]_{ \quad \sfN} & \big(\Cat \downarrow \St_S(L)\big)
}
$$
commutes up to an equivalence of 2-functors.  
\end{proposition}

\begin{proof}
The 2-functor $\St_{S''}(CL) \to \big(\Cat \downarrow \St_{S'}\big((0,1) \times L\big)\big)$  is inverse to the 2-functor that takes an object $(\bC,\scrC,\phi)$ to the unique stack given by the formula
$$U \mapsto
\bigg\{
\begin{array}{ll}
\bC & \text{if } U \text{ is of the form } C_\epsilon L = [0,\epsilon) \times L / \{0\} \times L \\
\scrC(U) & \text{if $U$ does not contain the cone point}
\end{array}
$$
\end{proof}

\begin{theorem}
\label{orpimonodromy}
Let $(X,S)$ be a topologically stratified space, and let $F$ be an exit 2-monodromy functor on $(X,S)$.  The prestack $\sfN F$ is an $S$-constructible stack, and the 2-functor $\sfN:\twoexit(X,S) \to \St_S(X)$ is an equivalence of 2-categories.
\end{theorem}

\begin{proof}
We will follow the proof of theorem \ref{twomonod}.  
Let $G$ be another exit 2-monodromy functor on $X$, and once again let $\sfN(G,F)$ be the prestack $U \mapsto \Hom_{\twoexit(U)}(G\vert_U,F\vert_U)$.  As in the proof of theorem \ref{twomonod}, the van Kampen property of $\orpi_{\leq 2}$ (axiom (vK)) implies $\sfN(G,F)$ is a stack.  By the homotopy axiom (H), $\orpi_{\leq 2}(V) \to \orpi_{\leq 2}(U)$ is an equivalence of 2-categories whenever $V \subset U$ are open sets and $V \hookrightarrow U$ is a loose stratified homotopy equivalence.  It follows that the stacks $\sfN(G,F)$ are constructible by theorem \ref{cscrit}.  In particular $\sfN F$ is a constructible stack.

To see that $\sfN:\twoexit(X,S) \to \St_S(X)$ is essentially fully faithful, it suffices to show that
$\sfN(G,F) \to \shHom(\sfN G, \sfN F)$ is an equivalence of stacks, and we may check this on stalks.  We will induct on the dimension of $X$: it is clear that this morphism is an equivalence of stacks when $X$ is 0-dimensional, so suppose we have proven it an equivalence for $X$ of dimension $\leq d$.  Let $x \in X$ and let $U$ be a conical neighborhood of $x$.  The morphism $\sfN(G,F)_x \to \shHom(\sfN G, \sfN F)_x$ is equivalent to the morphism $\sfN(G,F)(U) \to \Hom(\sfN G \vert_U, \sfN F\vert_U)$, and by the stratified homotopy equivalence $U \simeq CL$ we may as well assume $U = CL$.  Let $T$ denote the stratification on $L$.  By proposition \ref{twomonforcones}, we have to show that the 2-functor $\big(\Cat \downarrow \twoexit(L,T)\big) \to \big(\Cat \downarrow \St_T(L)\big)$ is an equivalence, but this map is induced by $\sfN:\twoexit(L,T) \to \St_T(L)$ which is an equivalence by induction.

Finally let us show that $\sfN:\twoexit(X,S) \to \St_S(X)$ is essentially surjective.  Again let us induct on the dimension of $X$.  Let $\scrC$ be a constructible stack on $X$.  The restriction of $\scrC$ to a conical open set $U \cong \bbR^d \times CL$ is in the essential image of $\sfN:\twoexit(U,S_U) \to \St_{S_U}(U)$ by induction and proposition \ref{twomonforcones}.  We may find a $d$-cover $\{U_i\}_{i \in I}$ of $X$ generated by conical open sets, so that for each $i$ there is an $F_i \in \twoexit(U_i)$ such that $\scrC\vert_{U_i}$ is equivalent to $\sfN F_i$.  These $F_i$ assemble to an exit 2-monodromy functor on the $d$-cover, which by axiom (vK) comes from an exit 2-monodromy functor $F$ on $X$ with $\sfN F \cong \scrC$.  

This completes the proof.

\end{proof}

\section{Exit paths in a stratified space}
\label{sec7}

In this section we identify a particular stratified 2-truncation: the exit-path 2-category $EP_{\leq 2}$.  If $(X,S)$ is a topologically stratified space, then the objects of $EP_{\leq 2}(X,S)$ are the points of $X$, the morphisms are Moore paths with the ``exit property'' described in the introduction, and the 2-morphisms are homotopy classes of homotopies between exit paths, subject to a tameness condition.  The purpose of this section is to give a precise definition of the functor $EP_{\leq 2}$, and to check the axioms \ref{2truncaxioms}.

\begin{definition}
Let $X$ be a topologically stratified space.  A path $\gamma:[a,b] \to X$ is called an \emph{exit path} if for each $t_1, t_2 \in [a,b]$ with $t_1 \leq t_2$, the point $\gamma(t_1)$ is in the closure of the stratum containing $\gamma(t_2)$; equivalently, if the dimension of the stratum containing $\gamma(t_1)$ is not larger than the dimension of the stratum containing $\gamma(t_2)$.  For each pair of points $x,y \in X$ let $EP(x,y)$ denote the subspace of the space $P(x,y)$ of Moore paths (section \ref{fun2gpd}) with the exit property, starting at $x$ and ending at $y$.  
\end{definition}

\begin{remark}
If we wish to emphasize the space $X$ we will sometimes write $EP(X;x,y)$ for $EP(x,y)$.
\end{remark}

\subsection{Tame homotopies}

Let $(X,S)$ be a topologically stratified space.  Let us call a map $[0,1]^n \to X$ \emph{tame} with respect to $S$ if there is a continuous triangulation of $[0,1]^n$ such that the interior of every simplex maps into a stratum of $X$.  Note that the composition of a tame map $[0,1]^n \to X$ with a stratum-preserving map $(X,S) \to (Y,T)$ is again tame.

If $x$ and $y$ are two points of $X$, call a path $h:[0,1] \to EP(x,y)$ tame if the associated homotopy $[0,1] \times [0,1] \to X$ is tame with respect to $S$.  (See remark \ref{moorehot} for how to associate an ordinary ``square'' homotopy to a homotopy between Moore paths.)  Finally if $H:[0,1] \times [0,1]$ is a homotopy between paths $h$ and $g$ in $EP(x,y)$, we call $H$ tame if the associated map $[0,1] \times [0,1] \times [0,1] \to X$ is tame with respect to $S$.  

\begin{definition}
Let $(X,S)$ be a topologically stratified space, and let $x$ and $y$ be points of $X$.  Let $\tame(x,y)$ be the
groupoid whose objects are the points of $EP(x,y)$ and whose hom sets $\Hom_{\tame(x,y)}(\alpha,\beta)$ are tame homotopy classes of tame paths $h:[0,1] \to EP(x,y)$ starting at $\alpha$ and ending at $\beta$.
\end{definition}

The concatenation map $EP(y,z) \times EP(x,y) \to EP(x,z)$ takes a pair of tame homotopies $h:[0,1] \to EP(x,y)$ and $k:[0,1] \to EP(y,z)$ to a tame homotopy $k\cdot h:[0,1] \to EP(x,z)$, and this gives a well-defined functor
$\tame(y,z) \times \tame(x,y) \to \tame(x,z)$.  It follows we may define a 2-category:

\begin{definition}
Let $(X,S)$ be a topologically stratified space.  Let $EP_{\leq 2}(X,S)$ denote the 2-category whose objects are points of $X$ and whose hom categories $\Hom_{EP_{\leq 2}(X,S)}(x,y)$ are the groupoids $\tame(x,y)$.
\end{definition}

\begin{remark}
\label{remtameness}
The tameness condition is necessary for our proof of the van Kampen property of $EP_{\leq 2}(X,S)$ (which follows the proof given in section \ref{vksubsec}) -- it allows us to subdivide our homotopies indefinitely.  We can define a similar 2-category $EP^{\mathrm{naive}}_{\leq 2}(X,S)$ whose hom categories are the fundamental groupoids $\pi_{\leq 1}\big(EP(x,y)\big)$.  I \emph{believe} this 2-category to be naturally equivalent to $EP_{\leq 2}(X,S)$, and that $EP_{\leq 2}^{\mathrm{naive}}$ could be used in place of $EP_{\leq 2}$ in our main theorem.  To prove this one would have to show that the natural functor $\tame(x,y) \to \pi_{\leq 1}\big( EP(x,y)\big)$ is an equivalence of groupoids.  I have been unable to obtain such a ``tame approximation'' result.
\end{remark}

\subsection{The exit path 2-category is a stratified 2-truncation}

As a stratum-preserving map $f:(X,S) \to (Y,T)$ preserves tameness of maps $[0,1]^n \to X$, it induces a functor $f_*:\tame(x,y) \to \tame\big( f(x), f(y)\big)$ and a strict 2-functor $f_*:EP_{\leq 2}(X,S) \to EP_{\leq 2}(Y,T)$.  Thus, $EP_{\leq 2}$ is a functor $\Strat \to \twocat$.  The remainder of this section is devoted to showing that $EP_{\leq 2}$ satisfies the axioms \ref{2truncaxioms} for a stratified 2-truncation.

\begin{theorem}
\label{epnh}
The 2-functor $EP_{\leq 2}:\Strat \to \twocat$ satisfies axioms (N) and (H) of \ref{2truncaxioms}.
\end{theorem}

\begin{proof}
Clearly $EP_{\leq 2}(\emptyset)$ is empty, so $EP_{\leq 2}$ satisfies axiom (N).

Let us now verify axiom (H).  Let $(X,S)$ be a topologically stratified space, and let $\pi:(0,1) \times X \to X$ denote the projection map.  
The 2-functor $\pi_*:EP_{\leq 2}\big( (0,1) \times X \big) \to EP_{\leq 2}(X)$ is clearly essentially surjective.  Let $(s,x)$ and $(t,y)$ be two points in $(0,1) \times X$.  To show that $\pi_*$ is essentially fully faithful we have to show that 
$tame\big( (s,x),(t,y) \big) \to \tame(x,y)$ is an equivalence of groupoids.  In fact this map is equivalent to the projection $\tame(s,t) \times \tame(x,y) \to \tame(x,y)$, and the groupoid $\tame(s,t)$ coincides with the fundamental groupoid $\pi_{\leq 1}\big(EP(s,t)\big)$ as $(0,1)$ has a single stratum.  Since $EP(s,t)$ is contractible, this groupoid is equivalent to the trivial groupoid, so the projection $\tame(s,t) \times \tame(x,y) \to \tame(x,y)$ is an equivalence.

\end{proof}

\begin{theorem}
\label{epc}
The 2-functor $EP_{\leq 2}:\Strat \to \twocat$ satisfies axiom (C) of \ref{2truncaxioms}
\end{theorem}

\begin{proof}
Let $(L,S)$ be a compact topologically stratified space.  Let $CL$ be the open cone on $L$, and let $* \in CL$ be the cone point.  We have to show that for each $x \in CL$, the groupoid $\tame(*,x)$ is equivalent to the trivial groupoid.  This is clear when $x$ is the cone point, so suppose $x  = (u,y) \in (0,1) \times L \subset CL$.  Let $\tame' \subset \tame(*,x)$ denote the full subgroupoid whose objects are the exit paths $\alpha$ of Moore length 1 (i.e. $\alpha:[0,1] \to CL$) with $\alpha(t) \neq *$ for $t>0$.  Every exit path $\gamma \in \tame(x,y)$ is clearly tamely homotopic to one in $\tame'$; it follows that $\tame'$ is equivalent $\tame(*,x)$.

Let $W \subset EP(*,x)$ be the subspace of exit paths $\alpha$ with Moore length 1 and with $\alpha(t) \neq *$ for $t >0$.  $W$ is homeomorphic to the space of paths $\beta:(0,1] \to (0,1) \times L$ with the property that $\beta$ is an exit path, that $\beta(1) = x$, and that for all $\epsilon>0$, there is a $\delta>0$ such that $\beta^{-1}\big( (0,\epsilon) \times L \big) \supset (0,\delta)$.  This space may be expressed as a product $W \cong W_1 \times W_2$, where
\begin{enumerate}
\item $W_1$ is the space of paths $\alpha:(0,1] \to (0,1)$ with $\alpha(1) = u$ and $\forall \epsilon \exists \delta$ such that $\alpha^{-1}(0,\epsilon) \supset (0,\delta)$
\item $W_2$ is the space of paths $\beta:(0,1] \to L$ with $\beta(1) = y$ and $\beta$ has the exit property.
\end{enumerate}
The first factor $W_1$ is contractible via $\kappa_t:W_1 \to W_1$, where $\kappa_t(\alpha)(s) = t \cdot s \cdot u + (1-t) \cdot \alpha(s)$.  The second factor is contractible via $\mu_t:W_2 \to W_2$ where $\mu_t(\beta)(s) = \beta(t+s - ts)$.  These contractions preserve tameness, and therefore they induce an equivalence between $\tame'$ and the trivial groupoid.
\end{proof}

Finally we have to prove that $EP_{\leq 2}$ satisfies the van Kampen axiom.  Let us first discuss elementary tame homotopies, analogous to the elementary homotopies used in the proof of the van Kampen theorem for $\pi_{\leq 2}$ in section \ref{vksubsec}.

\begin{definition}
Let $(X,S)$ be a topologically stratified space, and let $\{U_i\}_{i \in I}$ be a $d$-cover of $X$.  Let $x$ and $y$ be points of $X$, and let $\alpha$ and $\beta$ be exit paths from $x$ to $y$.  A homotopy $h:[0,1] \times [0,1] \to X$ between $\alpha$ and $\beta$ is \emph{$i$-elementary} if there is a subinterval $[a,b] \subset [0,1]$ such that $h(s,t)$ is independent of $s$ so long as $t \notin [a,b]$, and such that the image of $[0,1] \times [a,b] \subset [0,1] \times [0,1]$ under $h$ is contained in $U_i$.
\end{definition}

\begin{remark}
Elementary homotopies between exit paths may be pictured in the same way as ordinary homotopies, as in figure \ref{bead}.
\end{remark}

\begin{proposition}
\label{prop97}
Let $(X,S)$ be a topologically stratified space, and let $\{U_i\}_{i \in I}$ be a $d$-cover of $X$.  Let $\alpha$ and $\beta$ be exit paths from $x$ to $y$, and let $h:[0,1] \times [0,1] \to X$ be a homotopy from $\alpha$ to $\beta$.  Then there is a finite list $\alpha = \alpha_0$,$\alpha_1$,$\ldots$,$\alpha_n = \beta$ of exit paths from $x$ to $y$, and of homotopies $h_1:\alpha_0 \to \alpha_1$, $h_2:\alpha_1 \to \alpha_2$, $\ldots$, $h_n:\alpha_{n-1} \to \alpha_n$ such that $h$ is homotopic to $h_n \circ \ldots \circ h_1$, and such that each $h_i$ is elementary.
\end{proposition}

\begin{proof}
As in the proof of proposition \ref{elhotfactor}, it suffices to find a suitable triangulation of $[0,1] \times [0,1]$.  In our case a triangulation is ``suitable'' if each triangle is mapped into one of the charts $U_i$, and if furthermore for each triangle $\sigma$ we may order the vertices $v_1, v_2, v_3$ in such a way that $h$ carries the half-open line segment $\overline{v_1 v_2} - v_1$ into a stratum $X_k$, and the third-open triangle $\overline{v_1 v_2 v_3} - \overline{v_1 v_2}$ into a stratum $X_\ell$.  In that case we may find a parameterization $g:[0,1] \times [0,1] \to \sigma$ of $\sigma$ with the property that for each $t$ the path $[0,1] \to \{t\} \times [0,1] \to \sigma \to X$ has the exit property in $X$.  We may find a triangulation with these properties by picking a triangulation that is fine enough with respect to $\{U_i\}$, and taking its barycentric subdivision.
\end{proof}

We also may discuss elementary 3-dimensional homotopies between homotopies between exit paths, and a version of proposition \ref{3delhotfactor} holds.

\begin{definition}
\label{def98}
Let $(X,S)$ be a topologically stratified space, and let $\{U_i\}_{i \in I}$ be a $d$-cover of $X$.  Let $x,y \in X$, $\alpha,\beta \in EP^M(x,y)$, and let $h_0,h_1:[0,1] \times [0,1] \to X$ be homotopies from $\alpha$ to $\beta$.  A homotopy $t \mapsto h_t$ between $h_0$ and $h_1$ is called \emph{$i$-elementary} if there is a closed rectangle $[a,b] \times [c,d] \subset [0,1] \times [0,1]$ such that
\begin{enumerate}
\item $h_t(u,v)$ is independent of $t$ for $(u,v) \notin [a,b] \times [c,d]$.
\item For each $t$, $h_t\big([a,b] \times [c,d]\big) \subset U_i$.
\end{enumerate}
\end{definition}

\begin{proposition}
\label{prop99}
Let $(X,S)$, $\{U_i\}$, $x,y,\alpha,\beta$ be as in definition \ref{def98}.  Let $h$ and $g$ be homotopies from $\alpha$ to $\beta$.  Suppose that $h$ and $g$ are homotopic.  Then there is a sequence $h = k_0, k_1,\ldots,k_n = g$ of homotopies from $\alpha$ to $\beta$ such that, for each $i$, $k_i$ is homotopic to $k_{i+1}$ via an elementary homotopy.
\end{proposition}

\begin{proof}
Similar to proposition \ref{3delhotfactor}.
\end{proof}

\begin{theorem}
\label{epvk}
The functor $EP_{\leq 2}:\Strat \to \twocat$ satisfies axiom (vK).
\end{theorem}

\begin{proof}
Let $X$ be a topologically stratified space, and let $\{U_i\}_{i \in I}$ be a $d$-cover of $X$.  We have to show that $\res:\twoexit(X) \to \twoexit\big(\{U_i\}\big)$ is an equivalence of 2-categories.  It suffices to show that $\res$ is essentially fully faithful and essentially surjective.  The proofs of these facts in the unstratified case -- propositions \ref{res2full}, \ref{resessfull}, and \ref{resesssurj} -- may be followed almost verbatim to obtain the same results, after substituting propositions \ref{prop97} and \ref{prop99} for propositions \ref{elhotfactor} and \ref{3delhotfactor}.
\end{proof}

\subsection{Proof of the main theorem}

We may now prove the main theorem stated in the introduction.

\begin{theorem}
\label{finaltheorem}
The 2-functor $EP_{\leq 2}$ is a stratified 2-truncation.  Because of this, for any topologically stratified space $(X,S)$ the 2-category of $S$-constructible stacks on $X$ is naturally equivalent to the 2-category of $\Cat$-valued 2-functors on $EP_{\leq 2}(X,S)$.
\end{theorem}

\begin{proof}
That $EP_{\leq 2}$ satisfies the axioms of a stratified 2-truncation is the content of theorems \ref{epnh}, \ref{epc}, and \ref{epvk}.  The conclusion that the main theorem holds is implied by theorem \ref{orpimonodromy}.
\end{proof}

\emph{Acknowledgements:}  I would like to thank Mark Goresky, Andrew Snowden, and Zhiwei Yun for many helpful comments and conversations.  I would especially like to thank my advisor, Bob MacPherson.  This paper is adapted from a Ph.D. dissertation written under his direction.

\appendix

\section{Stacks}

The word ``stack'' has at at least two different and related meanings in mathematics.  Maybe most frequently it refers to some kind of geometric object that represents a groupoid-valued, rather than a set-valued, functor.  But it may also refer to a sheaf of categories, where the sheaf structure and axioms have been modified to take account of the ``two-dimensional'' nature of categories -- that is, to take account of the fact that categories are most naturally viewed as the objects of a 2-category.  In this appendix we develop some basic properties of stacks in the second sense.

We will proceed in a way that emphasizes the similarity with sheaves.  It requires generalizing the basic definitions of category theory, such as limits and adjoint functors, to 2-categories.  We summarize what we need from the theory of 2-categories in appendix B.

\subsection{Prestacks}

\begin{definition}
A \emph{prestack} on a 2-category $\bI$ is a 2-functor $\scrC:\bI^{op} \to \Cat$.  A prestack on a topological
space $X$ is a prestack on the partially ordered set of open subsets of $X$, regarded as a 2-category whose objects are the open sets, whose 1-morphisms are the inclusion maps, and with trivial 2-morphisms.  In detail, a prestack $\scrC$ on a space $X$ consists of:
\begin{enumerate}
\item[(0)] An assignment that takes an open set $U$ to a category $\scrC(U)$.
\item[(1)] A contravariant assignment that takes an inclusion $V \subset U$ to a restriction functor
	    $\scrC(U) \to \scrC(V)$
\item[(2)] For each triple of open subsets $W \subset V \subset U$ an isomorphism between 
              $\scrC(U) \to \scrC(V) \to \scrC(W)$ and $\scrC(U) \to \scrC(W)$
\item[(3)] Such that the tetrahedron associated to each
              quadruple $Y \subset W \subset V \subset U$ commutes:
$$\tetrah{\scrC(U)}{\scrC(V)}{\scrC(W)}{\scrC(Y)}$$
\end{enumerate}
\end{definition}

Write $\Prest(\bI)$ for the 2-category of prestacks on a 2-category $\bI$, and $\Prest(X)$ for the 2-category of prestacks on a space $X$.

\begin{definition}
Let $\scrC$ be a prestack on $X$.  The \emph{stalk} of $\scrC$ at $x \in X$ is the category
$$\scrC_x =_{\mathit{def}} \wtwocolim{U \mid U \owns x} \scrC(U)$$
\end{definition}

\begin{remark}
Let $\bI$ be a 1-category.  The 2-category of 2-functors $\bI^{op} \to \Cat$ is equivalent to
the 2-category of so-called \emph{fibered categories} over $\bI$ (\cite{sga1}).  The theory of stacks is usually (\cite{sga1}, \cite{giraud}, \cite{vistoli}) developed using fibered categories rather than 2-functors.
\end{remark}

\subsection{Stacks}

A \emph{stack} is a prestack on $X$ that satisfies a kind of sheaf condition.  We find it convenient to phrase this condition in terms of 2-limits over an open cover; in order to make this precise we need our open covers to be closed under finite intersections.  We will call these ``descent covers'' or $d$-covers.

\begin{definition}
\label{def-dcover}
A \emph{$d$-cover} of a space $U$ is a subset $I \subset \text{Open}(U)$ of the set of open subsets of $U$ that is closed under finite intersections, and that covers $U$.
\end{definition}

Let $\scrC$ be a prestack on $X$.  If $U$ is an open subset of $X$ and $\{U_i\}_{i \in I}$ is a $d$-cover of $U$, then the restriction functors $\scrC(U) \to \scrC(U_i)$ assemble to a functor
$$\scrC(U) \to \wtwolim{I}{\scrC(U_i)}$$

\begin{definition}
Let $\scrC$ be a prestack on a space $X$.  Then $\scrC$ is a \emph{stack} if for each open set $U \subset X$ and each $d$-cover $\{U_i\}_{i \in I}$ of $U$, the restriction functor
$$\scrC(U) \to \wtwolim{I}{\scrC(U_i)}$$
is an equivalence of categories.  Let $\St(X) \subset \text{Prest}(X)$ denote the full subcategory of the 2-category of prestacks on $X$ whose objects are stacks.
\end{definition}

\begin{theorem}
\begin{enumerate}
\item
Let $\scrP$ and $\scrC$ be prestacks on $X$.  Suppose $\scrC$ is a stack.  The prestack $\shHom(\scrP,\scrC)$ on $X$ that takes an open set $U$ to the category $\Hom(\scrP\vert_U,\scrC\vert_U)$
is a stack.

\item
Let $\scrC$ be a stack on $X$.  Let $c$ and $d$ be two objects of $\scrC(X)$.  The presheaf of hom sets
$U \mapsto \Hom_{\scrC(U)}(c\vert_U,d\vert_U)$ is a sheaf.

\item
Let $\scrC$ and $\scrD$ be two stacks on $X$.  The map $\phi:\scrC \to \scrD$ has a left adjoint (resp. has a right adjoint, resp. is an equivalence) if and only if the maps $\phi_x:\scrC_x \to \scrD_x$ on stalks all have left adjoints (resp. all have right adjoints, resp. are all equivalences).

\item
The inclusion 2-functor $\St(X) \to \Prest(X)$ has a right adoint, called \emph{stackification}.  Denote the stackification of $\scrP$ by $\scrP^\dagger$.  The adjunction morphism $\scrP \to \scrP^\dagger$
induces an equivalence on stalks.

\end{enumerate}
\end{theorem}

\subsection{Operations on stacks}

\begin{definition}
Let $X$ and $Y$ be topological spaces, and let $f:X \to Y$ be a continuous map.  If $\scrC$ is a prestack on $X$ let $f_* \scrC$ denote the prestack on $Y$ that associates to an open set $U$ the category $f_* \scrC(U) := \scrC\big(f^{-1}(U)\big)$.  We call $f_*\scrC$ the \emph{pushforward} of $\scrC$. 
\end{definition}

It is easy to verify that $f_*$ defines a strict 2-functor $\Prest(X) \to \Prest(Y)$, and that if $\scrC$ is a stack then $f_*\scrC$ is also.  The definition for the pullback of a stack is more complicated -- it requires a direct  2-limit over neighborhoods, followed by stackification:

\begin{definition}
Let $X$ and $Y$ be topological spaces, and let $f:X \to Y$ be a continuous map.  If $\scrC$ is a prestack on $Y$, let $f_p^* \scrC$ be the prestack on $X$ that associates to an open set $U \subset X$ the category
$$f_p^* \scrC = \wtwocolim{V \mid V \supset f(U)}{\scrC(V)}$$
Let $f^* \scrC$ denote the stackification of the prestack $f_p^* \scrC$.
\end{definition}

\begin{example}
If $x \in X$ and $i:\{x\} \hookrightarrow X$ denote the inclusion map, the prestack $i_p^* \scrC$ on $\{x\}$ coincides with the stalk $\scrC_x$.  Thus, if $f: X \to Y$, we have an equivalence of stalk categories $\big( f_p^* \scrC\big)_x \cong \scrC_{f(x)}$.
\end{example}

\begin{example}
If $U \subset X$ is open, and $j:U \hookrightarrow X$ denotes the inclusion map, we have 
$j_p^* \scrC(V) = \scrC(V)$ when $V \subset U$ is another open set.  If $\scrC$ is a stack then $j_p^* \scrC$ is also.  We often denote $j^* \scrC$ by $\scrC \vert_U$.
\end{example}

\begin{proposition}
Let $X$ and $Y$ be topological spaces, and let $f:X \to Y$ be a continuous map.

The 2-functor $f_*:\St(X) \to \St(Y)$ is right 2-adjoint to the 2-functor $f^*:\St(Y) \to \St(X)$.

\end{proposition}

\subsection{Descent for stacks}
\label{descentforstacks}

Let $X$ be a topological space, and let $\{U_i\}_{i \in I}$ be a $d$-cover of $X$.  Suppose we are given the following data:

\begin{enumerate}
\item[(0)] For each $i \in I$ a stack $\scrC_i$ on $U_i$.
\item[(1)] For each $i, j \in I$ with $U_j \subset U_i$, an equivalence of stacks $\scrC_i \vert_{U_j} \stackrel{\sim}{\to} \scrC_j$.
\item[(2)] For each $i,j,k \in I$ with $U_k \subset U_j \subset U_i$, an isomorphism between the composite equivalence
$\scrC_i \vert_{U_j} \vert_{U_k} \stackrel{\sim}{\to} \scrC_j \vert_{U_k} \stackrel{\sim}{\to} \scrC_k$
and $\scrC_i \vert_{U_k} \stackrel{\sim}{\to} \scrC_k$.
\item[(3)] Such that for each $i,j,k,\ell \in I$ with $U_\ell \subset U_k \subset U_j \subset U_i$, the tetrahedron
commutes:
$$\tetrah{\scrC_i\vert_{U_j}\vert_{U_k}\vert_{U_\ell}}{\scrC_j \vert_{U_k} \vert_{U_\ell}}{\scrC_k \vert_{U_\ell}}{\scrC_\ell}$$
\end{enumerate}

We will abuse terminology and refer to such data as a \emph{stack on the $d$-cover $\{U_i\}$}.  Stacks on $\{U_i\}$ form the objects of a 2-category $\St(\{U_i\})$ in the natural way.  If $\scrC$ is a stack on $X$ then $\scrC_i := \scrC\vert_{U_i}$ and the identity 1- and 2-morphisms form a stack on the $d$-cover $\{U_i\}$, and the assignment $\scrC \mapsto \{\scrC_i := \scrC\vert_{U_i}\}$ forms a strict 2-functor in a natural way.

\begin{theorem}
The natural restriction 2-functor $\St(X) \to \St(\{U_i\})$ is an equivalence of 2-categories.
\end{theorem}

\begin{remark}
The 2-category $\St(\{U_i\})$ may be interpreted as an inverse limit (or ``inverse 3-limit'') of the 2-categories $\St(U_i)$.  There is a sense then in which the theorem means stacks form a \emph{2-stack}.  See \cite{breen2}.
\end{remark}

\section{2-categories}

In this appendix we summarize some of the theory of 2-categories, and fix our conventions.

\begin{definition}
A \emph{strict 2-category} $\bC$ consists of
\begin{enumerate}
\item a collection $\mathrm{Ob}(\bC)$ of \emph{objects}
\item for each pair $x,y \in \mathrm{Ob}(\bC)$
a category $\Hom_{\bC}(x,y)$
\item for each triple $x,y,z \in \mathrm{Ob}(\bC)$, a \emph{composition functor} $\Hom_\bC(y,z) \times \Hom_\bC(x,y) \to \Hom_\bC(x,z)$.  
\end{enumerate}
The composition functors are assumed to satisfy associativity and to have units
in the strict sense: for each object $x \in \bC$, there is an object $1_x$ of $\Hom(x,x)$ such
that for each $y$, $\Hom(x,y) \stackrel{\circ 1_x}{\longrightarrow} \Hom(x,y)$ and $\Hom(y,x) \stackrel{1_x \circ}{\longrightarrow} \Hom(y,x)$
are the identity functors, and such that the following diagram commutes
$$
\xymatrix{
\Hom(z,w) \times \Hom(y,z) \times \Hom(x,y) \ar[r] \ar[d] & \Hom(z,w) \times \Hom(x,z) \ar[d] \\
\Hom(y,z) \times \Hom(x,y) \ar[r] & \Hom(x,w)
}$$
for each $x,y,z,w$.
Objects of $\Hom_\bC(x,y)$ are called \emph{1-morphisms} of $\bC$, and morphisms of $\Hom_\bC(x,y)$ are called \emph{2-morphisms} of $\bC$.
\end{definition}

We will also use the following terminology:

\begin{definition}
A \emph{(2,1)-category} is a 2-category $\bC$ all of whose 2-morphisms are invertible.
\end{definition}

\begin{remark}
If $\cat$ denotes the cartesian closed category whose objects are categories, and whose morphisms are
functors, then a strict 2-category is a $\cat$-enriched category in the sense of \cite{enrichedcat}.  If $\gpd$ denotes the full subcategory of $\cat$ whose objects are groupoids, then a $(2,1)$-category is a $\gpd$-enriched category.
\end{remark}

\begin{example}
There is a 2-category $\Cat$ whose objects are categories, and where the usual categories of functors and natural transformations are the hom categories.  
\end{example}

Note that we use a different symbol for the 2-category $\Cat$ than for the 1-category $\cat$.  $\Cat$ is the more natural object.

\begin{remark}
There is a more natural notion of \emph{weak} 2-category, where the associativity diagram above is required to commute only up to isomorphism, and these isomorphisms are required to satisfy some equations of their own.  Every weak 2-category is equivalent in the appropriate sense to a strict one.  Moreover, the 2-categories encountered in this paper are either strict already (such as the 2-category of stacks or of prestacks) or else may be easily made strict by a trick (such as the fundamental 2-groupoid and the exit-path 2-category).  We have therefore decided to develop this paper in terms of strict 2-categories.
\end{remark}

\begin{definition}
Let $\bC$ be a 2-category.  The \emph{opposite 2-category} $\bC^{op}$ is the 2-category with the same
objects as $\bC$, with $\Hom_{\bC^{op}}(x,y) = \Hom_\bC(y,x)$, and with the evident composition functor.
\end{definition}

\begin{remark}
One could also define a kind of ``opposite 2-category'' by reversing only the 2-morphisms, or by reversing both 1- and 2-morphisms.  We will not need these variations and so we won't introduce notation for them.
\end{remark}

\subsection{Two-dimensional composition in a 2-category}
\label{twodcomposition}

Let $\bC$ be a 2-category.  Let $x$ and $y$ be objects in $\bC$.  If $\alpha$, $\beta$, and $\gamma$ are three 1-morphisms in $\bC$, and $f:\alpha \to \beta$ and $g:\beta \to \gamma$ are 2-morphisms, then we may of course form a third 2-morphism $g \circ f:\alpha \to \beta$ by taking the composition of $g$ and $f$ in the 1-category $\Hom_\bC(x,y)$.  

There is another direction that we may compose 2-morphisms.  Let $x$, $y$ and $z$ be three objects of $\bC$, let $\alpha$ and $\beta$ be two 1-morphisms from $x$ to $y$ and let $\gamma$ and $\delta$ be two 1-morphisms from $y$ to $z$.  Let $f:\alpha \to \beta$ and $h:\gamma \to \delta$ be 2-morphisms.  Then we may form a new 2-morphism $h \star f:\gamma \circ \alpha \to \delta \circ \beta$ by applying the functors $\gamma \circ (-):\Hom_\bC(x,y) \to \Hom_\bC(x,z)$ and $(-) \circ \beta:\Hom_\bC(y,z) \to \Hom_\bC(x,z)$.  That is, let $h \star f$ denote the composite map
$$\gamma \circ \alpha \stackrel{\gamma \circ f}{\to} \gamma \circ \beta \stackrel{h \circ \beta}{\to} \delta \circ \beta$$

Now, let $\bC$ be a 2-category, and let $x$, $y$, and $z$ be objects of $\bC$.  Let $\alpha$, $\beta$, and $\gamma$ be 1-morphisms $x \to y$, and let $\delta$, $\epsilon$, and $\zeta$ be 1-morphisms $y \to z$.  Let $f:\alpha \to \beta$, $g:\beta \to \gamma$, $h:\delta \to \epsilon$, and $k:\epsilon \to \zeta$ be 2-morphisms.  We have the following equation:
$$(k \star g) \circ (h \star f) = (k \circ h) \star (g \circ f)$$
In practice, this equation allows us to ignore the difference between $\circ$ and $\star$ for 2-morphisms.  In fact, it follows that given any collection of 2-morphisms that may be composed using $\circ$ and $\star$, any two compositions agree. 

\subsection{Adjoints and equivalences within a 2-category}

\begin{definition}
\label{def-internaladjoints}
Let $\bC$ be a 2-category.  Suppose $f$ is a 1-morphism between objects $x$ and $y$ in $\bC$.
A \emph{right adjoint} to $f$ is a triple $(g,\eta,\epsilon)$, where $g$ is a 1-morphism $y \to x$ called the \emph{adjoint},
$\eta:1_x \to gf$ and $\epsilon:fg \to 1_y$ are 2-morphisms called the \emph{adjunction morphisms}, and the so-called ``triangle identities'' hold:
the natural maps $\eta g: g \to gfg$ and $g\epsilon:gfg \to g$ compose to $1_g$, and the natural maps 
$f\eta:f \to fgf$ and $\epsilon f:fgf \to f$ compose to $1_f$.
Dually, $(f,\epsilon, \eta)$ is called a \emph{left adjoint} of $g$.  
\end{definition}

We sometimes abuse notation by suppressing the adjunction morphisms $\epsilon$ and $\eta$.

\begin{proposition}
Let $\bC$ be a 2-category, and let $f$ be a 1-morphism in $\bC$.  If $f$ has a right (resp. left) adjoint $(g,\alpha,\beta)$, then $(g,\alpha,\beta)$ is unique up to unique isomorphism commuting with $\alpha$ and $\beta$.
\end{proposition}

\begin{definition}
\label{def-internalequivalence}
Let $\bC$ be a 2-category, and let $f:x \to y$ be a 1-morphism in $\bC$.  Then $f$ is called an \emph{equivalence} if it has a right adjoint $g$, and if the adjunction maps $1 \to fg$ and $gf \to 1$ are both isomorphisms.  This is equivalent to requiring $f$ to have a left adjoint $g$ with isomorphisms for adjunction maps.
\end{definition}

\subsection{Commutative diagrams in a 2-category}
Recall that, in a 1-category, a commutative diagram is a collection of objects and of morphisms
between them such that any two paths of composable arrows in the diagram between objects
$x$ and $y$ coincide.  In a 2-category, we say that a diagram of objects, morphisms, and 2-morphisms commutes if for each pair of composable paths $f_1, f_2, f_3 \ldots$ and $g_1, g_2, g_3 \ldots$ between objects $x$ and$y$, any two composable sequences of 2-morphisms from $\ldots f_3 \circ f_2 \circ f_1$
to $\ldots g_3 \circ g_2 \circ g_1$ coincide.

\begin{example}
Let $\bC$ be a 2-category.  A \emph{commutative triangle} in $\bC$ is a triple $x,y,z$ of objects, a triple $x \to y$, $y \to z$, and $x \to z$ of 1-morphisms, and an isomorphism between the composite $x \to y \to z$ and the map $x \to z$.  A \emph{commutative square} is a tuple of objects $w,x,y,z$, tuple of 1-morphisms $w \to x$, $w \to y$, $x \to z$, $y \to z$, and a 2-isomorphism between the composites $w \to x \to z$ and $w \to y \to z$.
For typesetting reasons, we omit the picture of the isomorphism when we draw a commutative triangle or square:
$$
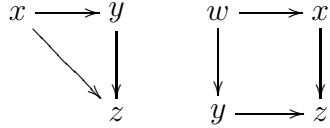
\xymatrix{x \ar[r] \ar[dr] & y \ar[d] & w \ar[r] \ar[d] & x \ar[d] \\
 & z & y \ar[r] & z}$$
 \end{example}

\begin{example}
Let $\bC$ be a 2-category.  A \emph{commutative tetrahedron} in $\bC$ is a tuple $w,x,y,z$ of objects, a collection of 1-morphisms $w \to x$, $w \to y$, $w \to z$, $x \to y$, $x \to z$, and a collection of 2-isomorphisms between the composite $w \to x \to y$ and $w \to y$, the composite $w \to x \to z$ and $w \to z$, the composite $w \to y \to z$ and $w \to z$, and the composite $x \to y \to z$ and $x \to z$, such that the two isomorphisms between the composite $w \to x \to z$ and $w \to y \to z$ coincide.  We often draw a commutative tetrahedron in the following manner:
$$\tetrah{w}{x}{y}{z}$$
\end{example}

\begin{remark}
\label{nerve}
We may go on defining commutative $n$-simplices for $n>3$.  
When we refer to a commutative $n$-simplex in $\bC$ we are referring to a diagram
in which all the 2-morphisms are isomorphisms.  This is most satisfying when $\bC$ is a
(2,1)-category: in that case the collection of
objects, 1-morphisms, commutative triangles, commutative tetrahedra, etc. assemble to a simplicial
set called the \emph{nerve} of the (2,1)-category.  (Defining the nerve of a general (2,2)-category is more subtle.)
\end{remark}

\begin{example}
A \emph{prism} with vertices $a,b,c,x,y,z$ is a collection of arrows $a \to b$, $a \to c$,
$b \to c$, $a \to x$ $b \to y$, $c \to z$, $x \to y$, $x \to z$, and $y \to z$ and a collection of 2-isomorphisms between $a \to b \to c$ and $a \to c$, between $x \to y \to z$ and $x \to z$, between $a \to b \to y$ and $a \to x \to y$, between $b \to c \to z$ and $b \to y \to z$, and between $a \to c \to z$ and $a \to x \to z$, as in the picture.

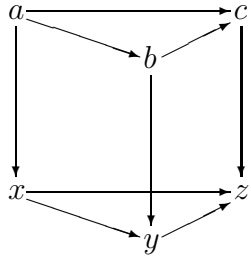
\begin{center}
\setlength{\unitlength}{.1cm}
\begin{picture}(40,35)

\put(10,30){\vector(0,-1){20}}
\put(40,30){\vector(0,-1){20}}
\put(28,23.5){\vector(0,-1){20}}

\put(11.5,32){\vector(1,0){27}}
\put(11.5,31){\vector(3,-1){15}}
\put(29.5,26){\vector(2,1){9}}

\put(11.5,8){\vector(1,0){27}}
\put(11.5,7){\vector(3,-1){15}}
\put(29.5,2){\vector(2,1){9}}

\put(9,31){$a$}
\put(39,31){$c$}
\put(27,24.5){$b$}
\put(9,7){$x$}
\put(39,7){$z$}
\put(27,.5){$y$}

\end{picture}
\end{center}
The prism is called commutative if the two isomorphisms between $a \to c \to z$ and $a \to x \to y \to z$ coincide.  
\end{example}

\subsection{2-groupoids}

\begin{definition}
A \emph{2-groupoid} is a 2-category in which all the morphisms are equivalences and all the 2-morphisms are isomorphisms.  
\end{definition}

\begin{remark}
A 2-groupoid with one object is called a \emph{2-group}.  To each object $x$ in a 2-category we can associate a 2-group $\mathrm{Aut}(x)$ whose unique object is $x$, whose morphisms are the self-equivalence of $x$, and whose 2-morphisms are isomorphisms between these self-equivalences.

It is possible to describe 2-groups in terms of more classical algebraic structures.
The nerve of a 2-groupoid is a fibrant simplicial set whose homotopy groups (in each connected component) vanish above dimension 2. 
It is possible to show that connected, pointed spaces $X$ whose homotopy groups vanish above dimension 2 are classified up to homotopy by a group $G$ (the fundamental group of $X$), a
commutative group $A$ (the second homotopy group of $X$), and an element in group cohomology
$H^3(G;A)$ (a Postnikov invariant of $X$).  All this is more naturally encoded in terms of a
``crossed module."
\end{remark}

\subsection{2-functors between 2-categories}

\begin{definition}
Let $\bC$ and $\bD$ be 2-categories.  A \emph{2-functor} $F:\bC \to \bD$ is
\begin{enumerate}
\item An assignment that takes an object $x\in \bC$ to an object $Fx \in \bD$.
\item A collection of functors $F:\Hom_\bC(x,y) \to \Hom_\bD(Fx,Fy)$.
\item For each triple $x,y,z$ of objects, a natural isomorphism $\mu_{x,y,z}$
between the two ways of composing the square:
$$\xymatrix{
\Hom_\bC(y,z) \times \Hom_\bC(x,y) \ar[r]^{(F,F) \quad \,\,\,\,\,\,\,} \ar[d] & \Hom_\bD(Fy,Fz) \times \Hom_\bD(Fx,Fy) \ar[d] \\
\Hom_\bC(x,z) \ar[r]_F & \Hom_\bD(Fx,Fz)}$$
\end{enumerate}
We furthermore assume that $F(1_x) = 1_{Fx}$ for each $x$, and that a certain diagram
of $\mu$s built from a quadruple $w,x,y,z$ of objects commutes.  See \cite{hakim} for details.
\end{definition}

\begin{definition}
Let $\bC$ and $\bD$ be 2-categories, and let $F:\bC \to \bD$ be a 2-functor.  $F$ is called \emph{strict}
if all the coherence maps $\mu_{x,y,z}$ are identities.  The data of a strict 2-functor is equivalent to the data of a $\cat$-enriched functor in the sense of \cite{enrichedcat}.
\end{definition}

\begin{remark}
We will shortly introduce a notion of equivalence for 2-functors (in fact we will introduce a 2-category of 2-functors); note that not all 2-functors are equivalent to strict ones.

Many authors reserve the word ``2-functor'' for what we have called strict functors, and call 2-functors
\emph{pseudofunctors} (e.g. \cite{gray}, \cite{kellystreet}), though usually not in the more recent literature.  
\end{remark}

\begin{remark}
There is a more general notion of 2-functor where the coherence maps $\mu$ are not required to be invertible.  They are often called ``lax 2-functors.''  There are lax versions of many concepts in 2-category theory, where one replaces an isomorphism in a definition with a map in one direction or another.  For our purposes -- that is, stacks of categories -- the non-lax versions of all these concepts seem to be the correct ones.
\end{remark}

\begin{example}
Let $\bC$ be a 2-category and let $\Cat$ be the 2-category of 1-categories.  For each object $x$ of $\bC$ there is a strict 2-functor $\Hom(x,-):\bC \to \Cat$ and a strict 2-functor $\Hom(-,x):\bC^{op} \to \Cat$.  In fact, there is a strict 2-functor $\Hom:\bC^{op} \times \bC \to \Cat$.
\end{example}

\begin{proposition}
Let $\bC$ and $\bD$ be 2-categories, and let $F:\bC \to \bD$ be a 2-functor. 
If $f$ is a 1-morphism in $\bC$, and $f$ has a left (resp. right) adjoint, then $Ff$ has a left (resp. right) adjoint in $\bD$.  Furthermore if $f$ is an equivalence in $\bC$ then $Ff$ is an equivalence in $\bD$.
\end{proposition}

\subsection{Composite 2-functors}
Let $\bC$, $\bD$, and $\bE$ be 2-categories.  Let $F:\bC \to \bD$ and $G:\bD \to \bE$ be 2-functors.
There is a \emph{composite 2-functor} $GF = G \circ F$ from $\bC$ to $\bE$; $GF$ is defined on objects,
1-morphisms, and 2-morphisms of $\bC$ in the evident way.  
To complete the definition it is necessary to describe the coherence data $\mu_{x,y,z}$ for $GF$ -- this is straightforward but we refer to \cite{gray} for details.

\subsection{The 2-category $\twofunct(\bC,\bD)$}
 
Let $\bC$ and $\bD$ be 2-categories.  For every pair of 2-functors $F:\bC \to \bD$ and $G:\bC \to \bD$
we may define a 1-category $\twonat(F,G)$.  Objects of $\twonat(F,G)$ are called \emph{2-natural transformations}, and morphisms are called \emph{modifications}.  
Given three 2-functors $F,G$ and $H$ from $\bC$ to $\bD$, a composition functor $\twonat(G,H) \times \twonat(F,G) \to \twonat(F,H)$ is defined.  This composition is strictly associative, and it has strict units, so this data defines a 2-category $\twofunct(\bC,\bD)$.  We define 2-natural transformations here; more details may be found in \cite{hakim}:

\begin{definition}
Let $\bC$ and $\bD$ be 2-categories.  Let $F$ and $G$ be two 2-functors from $\bC$ to $\bD$.
A \emph{2-natural transformation} $n$ from $F$ to $G$ consists of
\begin{enumerate}
\item An assignment that takes objects $x$ of $\bC$ to 1-morphisms $n(x):Fx \to Gx$ in $\bD$ 

\item An assignment that takes 1-morphisms $f:x \to y$ in $\bC$ to isomorphisms 
$n(f):Gf \circ n(x) \cong n(x) \circ Ff$.
\end{enumerate}
such that, for each pair of 1-morphisms $x \stackrel{f}{\to} y$ and $y \stackrel{g}{\to}
z$ in $\bC$, a certain prism with vertices $Fx,Fy,Fz,Gx,Gy,Gz$ commutes.  See \cite{hakim} for details.
\end{definition}

\begin{proposition}
Let $\bC$ and $\bD$ be 2-categories.  Let $n:F \to G$ be a 1-morphism in $\twofunct(\bC,\bD)$.
\begin{enumerate}
\item $n$ has a right (resp. left) adjoint $m:F \to G$ if and only if each 1-morphism $n(x):F(x) \to G(x)$ in $\bD$ has a left (resp. right) adoint. (see definition \ref{def-internaladjoints})

\item $n$ is an equivalence if and only if each $n(x):F(x) \to G(x)$ is an equivalence in $\bD$. (see definition \ref{def-internalequivalence}) 
\end{enumerate}

\end{proposition}

\subsection{Adjoint 2-functors between 2-categories}
\label{Badjoint}

Let $\bC$ and $\bD$ be 2-categories, and let $F:\bC \to \bD$ and $G:\bD \to \bC$ be 2-functors.  We may
form the two 2-functors
$$\begin{array}{c}
\Hom_{\bD}(F-,-):  \bC^{op} \times \bD  \to  \Cat \\
\Hom_{\bC}(-,G-): \bC^{op} \times \bD  \to  \Cat
\end{array}$$

\begin{definition}
Let $\bC$ and $\bD$ be 2-categories.  A \emph{2-adjunction} from $\bC$ to $\bD$ is a pair of 2-functors
$F:\bC \to \bD$ and $G:\bD \to \bC$ together with a 2-natural equivalence -- i.e. an equivalence in the 2-category $\twofunct(\bC^{op}\times\bD,\Cat)$ -- between the two 2-functors
$\Hom(F-,-)$ and $\Hom(-,G-)$.

We say that $F$ is \emph{left 2-adjoint} to $G$ and that $G$ is \emph{right 2-adjoint} to $F$.
\end{definition}

\begin{remark}
This is another definition with lax generalizations.
\end{remark}

We can apply the adjunction to the identity 2-functors $Fc \to Fc$ and $Gd \to Gd$ to obtain 2-natural transformations $1_\bC \to GF$ and $FG \to 1_\bD$.

\begin{definition}
Let $\bC$ and $\bD$ be 2-categories.  An \emph{equivalence} from $\bC$ to $\bD$ is a pair of adjoint 2-functors $F:\bC \to \bD$, $G:\bD \to \bC$ such that the 2-natural transformations $1_\bC \to GF$ and $FG \to 1_\bD$ are equivalences.
\end{definition}

A 2-functor $F:\bC \to \bD$ is called \emph{essentially fully faithful} if the functor $\Hom_\bC(x,y) \to \Hom_\bD(Fx,Fy)$ is an equivalence of categories for every pair of objects $x,y \in \bC$.  $F$ is called \emph{essentially surjective} if for every object $d \in \bD$ there is an object $c \in \bD$ such that $Fc$ and $d$ are equivalent in $\bD$.

\begin{proposition}
Let $\bC$ and $\bD$ be 2-categories.  Let $F:\bC \to \bD$ be a 2-functor.  The following are equivalent.
\begin{enumerate}
\item $F$ is essentially fully faithful and essentially surjective.
\item $F$ is part of an equivalence $F:\bC \to \bD$, $G:\bD \to \bC$.
\end{enumerate}
\end{proposition}

\subsection{Direct and inverse 2-limits of 1-categories}
\label{Blimits}

Let $\bI$ be a 2-category, and let $F:\bI \to \Cat$ be a 2-functor.  It is possible to define categories $\twolim_\bI F$ and $\twocolim_\bI F$ with the appropriate 2-universal property.  We will give these definitions here in the form that we need them; in particular we only define $\twocolim_\bI F$ in case $\bI$ is a filtered poset.

\begin{definition}
Let $\bI$ be a 2-category and let $F:\bI \to \Cat$ be a 2-functor.  Let $\twolim_I F$ denote the category whose objects consist of the following data:
\begin{enumerate}
\item An assignment that takes an object $i \in \bI$ to an object $x_i \in F(i)$.
\item An assignment that takes a morphism $f:i \to j$ to an isomorphism $F(f)(x_i) \cong x_j$.
\end{enumerate}
We require that for each commutative triangle $\alpha:gf \cong h$ with vertices $i,j,k$ in $\bI$,
the diagram
$$\xymatrix{
F(g)F(f)x_i \ar[r]\ar[d]^{\alpha} & F(g) x_j \ar[d] \\
F(h) \ar[r] & x_k}
$$
commutes.
The morphisms of this category are collections of maps $x_i \to y_i$ commuting with the maps
$F(f)(x_i) \to x_j$.
\end{definition}

\begin{definition}
Let $I$ be a filtered poset, and let $F:I \to \Cat$ be a 2-functor.  Let $\twocolim_I F$ denote
the category whose objects are $\coprod_{i \in I} F(i)$, and whose morphisms 
$\Hom(x \in F(i),y \in F(j))$ are elements of the limit
$$\varinjlim_{k\geq i \text{ and } j} \Hom_{F(k)}(x_i\vert_k,x_j\vert_k)$$
Here if $\ell \leq k$ and $x \in F(\ell)$, the notation $x\vert_k$ denotes the image of $x$ under the functor $F(\ell) \to F(k)$ induced by the unique morphism $\ell \to k$.
\end{definition}

\end{document}